\RequirePackage{fix-cm}
\documentclass[smallextended]{svjour3}       % onecolumn (second format)
\smartqed  % flush right qed marks, e.g. at end of proof

\usepackage{bbm}
\usepackage{times}
\usepackage{amsmath}
\usepackage{amsmath}
\usepackage{amssymb}
\usepackage{amsfonts}
\usepackage{hyperref}
\usepackage{tikz}
\usepackage{xargs}
\usepackage{mathtools}
\usepackage{graphicx}
\usepackage{natbib}
\usepackage{algorithm}
\usepackage{algorithmic}
\usepackage[margin=1.25in]{geometry}
\newtheorem{assumption}[theorem]{Assumption}

\newcommand{\R}{\mathbb R}
\newcommand{\Z}{\mathbb Z}

\newcommand{\C}{\mathbb C}
\newcommand{\nats}{\mathbb N}

\def\bI{\boldsymbol{I}}

\def\bR{\boldsymbol{R}}

\def\bA{\boldsymbol{A}}

\def\bS{\boldsymbol{S}}

\def\bC{\boldsymbol{C}}

\def\bU{\boldsymbol{U}}

\def\bv{\boldsymbol{v}}

\def\bx{\boldsymbol{x}}

\def\by{\boldsymbol{y}}
\def\bz{\boldsymbol{z}}

\def\bu{\boldsymbol{u}}
\def\bb{\boldsymbol{b}}

\def\cR{\mathcal{R}}

\def\cS{\mathcal{S}}
\def\cR{\mathcal{R}}

\def\cD{\mathcal{D}}

\def\cX{\mathcal{X}}
\def\cY{\mathcal{Y}}
\def\cT{\mathcal{T}}

\def\cH{\mathcal{H}}
\def\cZ{\mathcal{Z}}

\def\bzero{\boldsymbol{0}}

\def\bxi{\boldsymbol{\xi}}

\def\bDelta{\boldsymbol{\Delta_s}}

\def\argmin{\text{argmin}}
\def\argmax{\text{argmax}}

\def\ave{\text{ave}}

\DeclareMathOperator{\Tr}{\text{Tr}}

\DeclareMathOperator{\SO}{\text{SO}}
\DeclareMathOperator{\depth}{\mathsf{depth}}

\DeclareMathOperator{\Exp}{\text{Exp}}
\DeclareMathOperator{\Log}{\text{Log}}

\newcommand{\bset}[1]{\Big\{#1\Big\}}
\begin{document}

\title{Depth Descent Synchronization in $\SO(D)$}

\author{Tyler Maunu         \and
	Gilad Lerman%etc.
}


\institute{Tyler Maunu \at
	Department of Mathematics, Massachusetts Institute of Technology\\
	\email{maunut@mit.edu}           %  \\
	\and
	Gilad Lerman \at
	School of Mathematics, University of Minnesota \\
	\email{lerman@umn.edu}
}

\date{}

%\date{\today}

%%%%%%%%%%%%%%%%%%%%%%%%%%%%%%%%%%%%%%%%%%%%%%%%%%%%%%%%%%%%%%%%%%%%%%%%%%%%%%%%
%uncomment the following when switching to colt
%\title[Provably Robust Multiple Rotation Averaging for $\SO(2)$]{A Provably Robust Multiple Rotation Averaging Scheme for $\SO(2)$}

%\coltauthor{%
%	\Name{Tyler Maunu} \Email{maunut@mit.edu}\\
%	\addr Department of +Mathematics \\ Massachusetts Institute of Technology \\ Cambridge, MA 02139 
%	\AND
%	\Name{Gilad Lerman} \Email{lerman@umn.edu}\\
%	\addr School of Mathematics \\ University of Minnesota \\ Minneapolis, MN 55455
%}

\maketitle

%\tableofcontents

 \begin{abstract}
	 We give robust recovery results for synchronization on the rotation group, $\SO(D)$. In particular, we consider an adversarial corruption setting, where a limited percentage of the observations are arbitrarily corrupted. We develop a novel algorithm that exploits Tukey depth in the tangent space of $\SO(D)$. This algorithm, called Depth Descent Synchronization, exactly recovers the underlying rotations up to an outlier percentage of $1/(D(D-1)+2)$, which corresponds to $1/4$ for $\SO(2)$ and $1/8$ for $\SO(3)$. In the case of $\SO(2)$, we demonstrate that a variant of this algorithm converges linearly to the ground truth rotations. { We implement this algorithm for the case of $\SO(3)$ and demonstrate that it performs competitively on baseline synthetic data.}
 \end{abstract}

\section{Introduction}

The typical synchronization problem involves recovery of $n$ group elements from pairwise measurements between them. It arises, for example, when solving the Structure from Motion (SfM) problem. One subproblem of SfM is to recover the three-dimensional orientations and positions of cameras from pairwise orientations and positions in relation to a scene~\citep{ozyecsil2017survey}. Here, we specifically focus on robust synchronization over $\SO(D)$, the rotation group for $\R^D$. That is, given pairwise rotations in $\SO(D)$, some of which are corrupted, we aim to recover the original set of $n$ rotations.

We assume $n$ unknown, ground truth elements of $\SO(D)$, which we denote by $\bR_1^\star, \dots, \bR_n^\star$. We form a graph $G([n],E)$, where  $[n]:=\{1,\ldots,n\}$ indexes the $n$ unknown elements and $E$ designates the edges for which measurements of relative rotations are taken. For each $jk \in E$, we are provided with the measurement
\begin{equation}\label{eq:meas}
	\bR_{jk}^\star = \bR_j^\star \bR_k^{\star \top}.
\end{equation}
We can think of $\bR_{jk}^\star$ in the following way: If we are oriented in the coordinate system with respect to node $k$, then $\bR_{jk}^\star$ rotates our coordinate system into the coordinate system we would see if we were sitting at node $j$. 
This synchronization formulation extends to any given group, where one wishes to recover $(g_1, \dots, g_n)$, an $n$-tuple of elements in the group, given  measurements of the group ratios $g_i g_j^{-1},\ i,j=1,\dots,n$.

In reality, we cannot hope to exactly measure all the pairwise rotations in~\eqref{eq:meas}. In many real systems, both noisy and corrupted measurements occur: our focus here is on \emph{adversarially corrupted measurements}. That is, within the measurement graph $G$, the corruption model is assumed to be fully adversarial. Our model is specified by partitioning the measured data into two parts: 
\begin{enumerate}
	\item We observe corrupted (or ``bad'') edges $E_b \subset E$, where all edges in $E_b$ have a corresponding arbitrary corruption. The adversary is allowed to choose $E_b$ (and thus may to some degree influence the connectivity of $E \setminus E_b$) as well as the corrupted values $\bR_{jk}$ for $jk \in E_b$. For each node, the adversary is only allowed to corrupt a limited fraction of edges.
	\item The rest of the observed edges are uncorrupted (or ``good'') edges $E_g = E \setminus E_b$, where each edge in $E_g$ has an associated measurement given by~\eqref{eq:meas}.
\end{enumerate}

Theoretically guaranteed methods for robust synchronization are still lacking, especially in adversarial and nonconvex settings. The development of these methods is important because in practice measurements are usually quite corrupted, especially in applied problems like Structure from Motion~\citep{ozyecsil2017survey}.
The results we establish here are concerned with \emph{exact recovery}. That is, given a set of corrupted measurements, we wish to exactly recover $\bR_1^\star, \dots, \bR_n^\star$. We will show that this is possible for a nonconvex method even in the presence of a significant amount of arbitrary corruption.

Our method falls into the class of \emph{multiple rotation averaging algorithms}~\citep{govindu2004lie,martinec2007robust,hartley2013rotation}. These methods are effectively coordinate descent algorithms, which are present highly efficient algorithms for nonconvex programs that are notoriously hard to analyze. While their analysis is challenging, it is imperative to develop a theoretical understanding of these methods and their robust counterparts~\citep{hartley2011l1,chatterjee2017robust}. Moreover, as we discuss later, there are few robustness guarantees for group synchronization with adversarial corruption. Among the limited guarantees, none cover our model, and we thus make a significant contribution to this area.
This work is also of general appeal to the nonconvex optimization community since we are able to prove convergence results in the complex nonconvex landscape of robust multiple rotation averaging. Furthermore, some energy landscapes associated with this problem exhibit many local minima and spurious fixed points, which we are able to avoid with our new method.

\subsection{Contributions of This Work}

The main contributions of this work follow. 
\begin{enumerate}
    \item {As a warm-up, we develop an adversarially robust algorithm for synchronization in $\SO(2)$, which we call Trimmed Averaging Synchronization (TAS). In Theorem~\ref{thm:tasrecovery}, under a generic condition on the measurement graph $G$, which we call the ``well-connectedness'', and proper initialization, we show that it can tolerate a fraction of outliers per node that is bounded above by $1/4$. We further prove that it converges linearly for fully connected observation graphs.}
	\item To extend this result to $\SO(D)$, we develop a new algorithm that we call \emph{Depth Descent Synchronization (DDS)} based on Tukey depth in the tangent space of $\SO(D)$. {To our knowledge, this is the first application of a manifold version of Tukey depth in an applied setting.}
	\item Assuming well-connectedness and good initialization, the DDS algorithm exactly recovers an underlying signal in the presence of a significant amount of adversarial outliers. This result is given in Theorem~\ref{thm:sodrecovery} and is the first guarantee of robustness to adversarial corruption for a multiple rotation averaging algorithm. {This result extends elegantly to sparse random graphs, where we show that it achieves the information theoretic rate with respect to graph sparsity for Erd\"os-R\'enyi observation graphs in Section~\ref{subsec:assumpdiscuss}.}
	\item { We show that this algorithm can be efficiently implemented for $\SO(3)$. We run baseline experiments that show it performs competitively on some baseline synthetic data for $\SO(3)$ synchronization, which arises in the important application of Structure from Motion.}
	%\item We discuss issues arising for a common multiple rotation averaging algorithm. In particular, we demonstrate a ``poorly-tempered landscape" in general, where multiple rotation averaging schemes based on the least absolute deviations energy may fail in the presence of some adversarial outliers.
\end{enumerate}

While we carefully review related work later in Section \ref{sec:related}, we emphasize here our contributions in terms of the most relevant works. { Again, we emphasize that we study an efficient, nonconvex algorithm for rotation synchronization that has guarantees for adversarial outliers.}

A robustness result for $\SO(D)$ synchronization based on semidefinite programming is given in~\citet{wang2013exact}. However, the probabilistic model in this work is very restrictive (see details in Section \ref{sec:related}), and the proposed method is slow for large $n$. 

{
\cite{huang2019learning} use a truncated least squares framework to do robust rotation synchronization. The truncated least squares framework was originally proposed by~\cite{truncatedLS} in the context of translation synchronization, and sequentially filters those pairwise measurements that are furthest from the current estimated pairwise measurements. Following this, \cite{huang2019learning} show that this can be extended to rotation synchronization. Under an appropriate choice of a thresholding parameter, they demonstrate that it is possible to exactly recover the ground truth in the presence of outliers if a certain generic condition is satisfied. This method has two downsides. First, one must repeatedly compute the lowest eigenvectors of the graph connection Laplacian, which has a higher memory cost for dense graphs than DDS and may also have issues of numerical stability. Second, the bound on the fraction of outliers that they present is not clear in general settings since it depends on the $(1,\infty)$-norm of the pseudoinverse of the graph connection Laplacian. This quantity is hard to control, and so it is unclear how their bound scales with various parameters as well as what it would state for arbitrary outliers. On the other hand,~\cite{truncatedLS} guarantees success of the truncated least squares method for adversarial outliers when considering translation synchronization in 1-dimension, but we do not see how to extend these to the problem of rotation synchronization.

}

The only existing result for adversarial robust synchronization was recently given by \citet{lerman2019robust}. 
They propose a general method, called Cycle-Edge Message Passing (CEMP), for group synchronization that is guaranteed to be robust to adversarial corruption. However, their method uses information from 3-cycles, that is, triangles in the graph, and so it is less efficient than typical multiple rotation averaging schemes by an order of $n$ (the ratio between the number of triangles and the number of edges in the graph). Beyond this, multiple rotation averaging algorithms are also attractive because they are more memory efficient. A caveat to our current work is that our new method is not as efficient as previous multiple rotation averaging algorithms. In particular, we require the computation of a depth-based estimator, and so each rotation update has complexity $O(n_j^3\log(n_j))$ for $\SO(3)$, where $n_j$ is the number of neighbors of the node to be updated. Therefore, we do not claim that DDS is uniformly most efficient for adversarially robust synchronization in terms of time complexity, although it is more computationally efficient than CEMP for very sparse graphs. However, our depth descent method does still have the benefit of more efficient memory complexity than~\cite{lerman2019robust}.

Beyond computational efficiency, the theoretical guarantees are also different: we bound the ratio of corrupted edges, whereas \citet{lerman2019robust} bound the ratio of corrupted triangles. {The method of~\cite{lerman2019robust} degrades with extremely sparse graphs, due to the fact that they need to ensure that it contains sufficiently many triangles, whereas we require a well-connectedness condition of the graph $G$ that extends to extremely sparse cases.} Other well-connectedness conditions appear, sometimes implicitly, in works minimizing energy functions \citep{wang_singer_2013, hand2018exact, lerman2018exact, truncatedLS}. {Finally, CEMP is tailored to finding the corruption level in the graph, but it does not have complete guarantees for the recovery of the underlying rotations themselves.}

%At last, we mention that the absolute deviation energy function over $\SO(D)^n$ that we discuss later is similar to the one of \citet{maunu2019well} over the Grassmannian for the different problem of robust subspace recovery (RSR)~\citep{lerman2018overview}. However, the latter energy is ``well-tempered'' under a generic condition, which implies a recovery guarantee for an adversarial setting~\citep{maunu2019robust}. On the other hand, there are significant issues that arise in trying to prove such well-tempered results in a synchronization framework. In this work, we point out some possible difficulties that arise in establishing a well-tempered landscape for the absolute deviation energy over $\SO(D)^n$, and therefore we resort to our novel robust procedure to avoid local minima. Although there are theoretical issues in proving convergence for the least absolute deviations methods, it is hard to find real examples where they actually converge to spurious minima. This fact points to some deeper phenomenon at play in practical settings that merits further theoretical study.

\subsection{Notation}

Bold uppercase letters will be used to denote matrices, while bold lowercase letters will be used to denote vectors. 
For a set $\cX$ in a Hilbert space, the convex hull is denoted by $\mathrm{conv}(\cX)$.
The sphere in $\mathbb{R}^D$ is written as $S^{D-1}$. For an ordered tuple of $n$ rotations, $\bR_1, \dots, \bR_n \in \SO(D)$, we write $(\bR) = (\bR_1, \dots, \bR_n)$.

\subsection{Structure of the Rest of the Paper}

{
We now outline the structure of this paper.
First, we review related work in Section~\ref{sec:related}.  We then discuss the specific case of synchronization over $\SO(2)$ in Section \ref{sec:tas} and give a simple, adversarially robust algorithm, called Trimmed Averaging Synchronization. Following this, in Section \ref{sec:dds} we develop our novel Depth Descent Synchronization algorithm, which utilizes Tukey depth to yield robust updates. Coupled with this, we develop its theoretical guarantees of robustness and convergence.
Section~\ref{sec:exp} presents some baseline experiments demonstrating the practicality of our proposed method.
%We finish in Section \ref{sec:conclude} with some open directions. %Supplemental proofs are provided in the appendix. %
}

\section{Related Work}
\label{sec:related}

Interest in the synchronization problem has grown in recent years due to applications in computer vision and image processing, such as SfM~\citep{govindu2004lie,martinec2007robust,arie2012global,hartley2013rotation,tron2009distributed,ozyesil2015stable,boumal2016nonconvex}, cryo-electron microscopy~\citep{wang2013exact} and
Simultaneous Localization And Mapping
(SLAM)~\citep{rosen2019se}.

The most common formulation for solving
rotation and other group synchronization problems involve a non-convex least squares formulation that can be addressed by spectral methods~\citep{singer2011angular} or semidefinite  relaxation~\citep{bandeira2017tightness}.
On the other hand, the work of~\citet{wang2013exact} uses a semidefinite relaxation of a least absolute deviations formulation to obtain a robust estimate for SO$(d)$ synchronization. They prove recovery for the pure optimizer of this convex problem in a restricted setting. In this setting the full graph is complete,  every edge is corrupted with a certain probability $p$ (in the case of $\SO(2)$, they require that $p \leq 0.543$ and for $\SO(3)$ they require $p \leq 0.5088$) and the corrupted group ratios are distributed uniformly on $\SO(D)$.
In practice, they advocate using an alternating direction augmented Lagrangian to solve their optimization problem. 
One may also use methods like the Burer-Monteiro formulation~\citep{boumal2018deterministic}, although current guarantees require the rank of the semidefinite program to be at least $O(\sqrt{n})$, which results in storing iterates much larger than the underlying signal that is a vector of size $n$~\citep{waldspurger2018rank}. Another recent work tries to leverage a low-rank plus sparse decomposition for robust synchronization~\citep{arrigoni2018robust}. However, this work does not contain robustness guarantees.

\subsection{Robust Synchronization Methods}

For a survey of robust rotation synchronization, see~\citet{tron2016survey}. Some early works on rotation synchronization include \citet{govindu2001combining,govindu2006robustness,martinec2007robust}, with later follow-up works by~\citet{hartley2013rotation,chatterjee2013efficient,chatterjee2017robust}. The later works discuss some least absolute deviations based approaches to multiple rotation averaging which we will discuss later. 
For theoretical foundations on averaging rotations, one can consult~\cite{moakher2002means}. For foundational work on optimization on the manifold SO($d$), see~\citep{taylor1994minimization,arora2009learning}.

Robust multiple rotation averaging algorithms were studied in \citet{hartley2011l1} and \citet{hartley2013rotation}. There, the authors used a least absolute deviations formulation over SO(3) using successive averaging with a Weiszfeld algorithm and a gradient-based algorithm. The authors also give a counterexample that shows that local minima exist and thus the global minimum of their problem may be hard to find in general. However, the authors give no guarantee of convergence or recovery in any setting. %The gradient methods that we analyze later were discussed by these authors without any guarantee. 
{
Also, we have found that this method may suffer from suboptimal fixed points in general, which we analyze in more detail in a forthcoming work (see also Section 5 of~\cite{maunu2020depth}).
}

 %Indeed, to extend our results to higher dimensional settings, one would need to leverage the associated manifold structures.

\subsection{Adversarially Robust Synchronization}

 A work that does contain guarantees is that of \citet{lerman2019robust}, which considers a message-passing procedure that incorporates consistent information from cycles. This algorithm was guaranteed to be robust for the adversarial setting and applies to any compact group. Although its adversarial setting is very general, it requires a bound on the ratio of corrupted cycles per edge and not on the ratio of corrupted edges. Furthermore, the use of cycles results in a potentially more computationally intensive algorithm than the one in this work that only uses pairwise information.

Guarantees for exact recovery with adversarial, or partially adversarial, corruption appear in few other synchronization problems. The adversarial corruption in $\Z_2$ synchronization is very special since there is a single choice to corrupt a group ratio. 
Under a special probabilistic model, \citet{Z2Afonso} established asymptotic and probabilistic exact recovery for the SDP relaxation of the least squares energy function of $\mathbb{Z}_2$ synchronization. The model assumes that $G([n],E)$ is an Erd\"os-R\'enyi graph with probability $p$ of connection, edges are randomly corrupted with probability $q$ and $p \, (1-2q)^2 \leq 0.5$. 
%The former recovery result corresponds to the information-theoretic bound of~\citet{Z2Afonso2} for this random model.
\citet{hand2018exact} and~\citet{lerman2018exact} established asymptotic exact recovery under a probabilistic model for solutions of the different problem of location recovery from pairwise orientations. In this problem ratios of the Euclidean group are normalized to the sphere. They assume an i.i.d.~Gaussian generative probabilistic model for the ground truth locations and an Erd\"os-R\'enyi model for the graph $G([n],E)$ and further bounded the ratio of maximal degree of $G([n],E_b)$ over $n$. In both works, these bounds approach zero as $n$ approaches infinity, unlike the constant bound of this work.

{Robust permutation synchronization was studied by~\cite{huang2013consistent}, where they give a maximum corruption percentage of 1/4 in the case of fully connected observation graphs. \citet{truncatedLS} analyzed a robust algorithm for one-dimensional translation synchronization that uses a truncated least squares formulation. They show that their method achieves a maximum corruption percentage of 1/6 for fully connected graphs. However, their generic condition is rather complicated and in order to interpret it they must make the fully connected assumption and they also restrict the maximal degree of $G([n],E_b)$.
In both~\citep{huang2013consistent,truncatedLS}, the bounds degrade for sparser observation graphs.

The results of~\cite{truncatedLS} were extended to the problem of rotation synchronization in~\cite{huang2019learning}. Here, the authors show that a truncated least squares formulation for rotation synchronization can recover the underlying signal assuming a generic bound that includes the $(1,\infty)$-norm of the pseudoinverse of the graph connection Laplacian. However, this bound is hard to interpret in general.
}

\subsection{Synchronization in Other Settings}

In contrast to corrupted settings, some works have considered estimation in a noisy setting. \citet{bandeira2017tightness} study maximum likelihood estimation of the angular synchronization problem and show that the associated semidefinite relaxation is tight.  
More recently, message-passing algorithms have been used for maximum likelihood estimation in the Gaussian setting~\citep{perry2018message}.
Other recent results leverage multiple phases to obtain better results in noisy settings~\citep{gao2019multi}.
Minimax estimation under the squared loss over $\SO(2)$ is considered in \cite{gao2020exact}, and optimization methods for the squared error over subgroups over the orthogonal group are considered in~\cite{liu2020unified}.

Another related problem over $\SO(2)$ is the synchronization of Kuramoto oscillators. In particular, a primary question is the minimal graph connectivity requirement ensuring that the energy landscape is nice. The weakest known requirement is that every vertex is connected to at least 0.7889$n$ other vertices~\citep{lu2019synchronization}. The conjectured bound is 0.75$n$, which is reminiscent of the bound we require for local recovery with adversarial corruption over $\SO(2)$.

\subsection{Nonconvex Optimization}

Optimization problems cast over $\SO(D)$ are usually nonconvex. We can think of our method as attempting to solve a nonconvex problem over $\SO(D)$ as well. Therefore, our work also fits in with the growing body of work analyzing nonconvex energy landscapes and procedures~\citep{dauphin2014identifying,hardt2014understanding,jain2014iterative,netrapalli2014non,yi2016fast,zhang2018robust,ge2015escaping,lee2016gradient,arora2015simple,mei2018landscape,ge2016matrix,boumal2016nonconvex,sun2015nonconvex,SunQuWright_nonconvex_sphere_2015,lerman2017fast,cherapanamjeri2017thresholding,ma2018implicit,maunu2019well}.

\subsection{Tukey Depth}

Finally, we appeal to tangent space depth, that is, using Tukey depth \citep{tukey1974t6} in the tangent space of the manifold $SO(D)$, to create a provable robust method. Tangent space depth, for a general manifold, first appeared in \cite{mizera2002depth}, where the author proves existence and depth bounds for maximum tangent depth estimators. Earlier work on Tukey, or halfspace, depth includes~\cite{rado1946theorem}, which proves a depth lower bound for general measures, and \cite{danzer1963helly}, which discusses the relation to Helly's Theorem. More recently, the classical reference of \cite{donoho1992breakdown} proves bounds on the maximum depth achieved in a dataset under ellipticity conditions. Computation of depth contours was considered in~\cite{liu2017fast,hammer2020estimating}. Recently, an interesting connection between depth estimators and generative adversarial networks has been exhibited~\citep{gao2018robust}, which may perhaps lead to more computationally efficient estimators.

\subsection{Notions of Robustness}
\label{subsec:robnotions}

We finish by clarifying our setting in the context of robustness. In order to quantify our notion of robustness, we introduce the following terminology. Recall that we have an underlying graph $G([n],E)$ corresponding to the pairwise measurements, where $E$ is partitioned into an inlier set, $E_g$, and outlier set $E_b$.
For any $j \in [n]$, we define its neighborhood as well as its inlier and outlier neighborhoods as
\begin{align}
	E^j = E_g^j \cup E_b^j, \quad E_g^j := \{k \in [n] : jk \in E_g\},\quad E_b^j := \{k \in [n] : jk \in E_b\}.
\end{align}
%The number of neighbors of an index $j$ is $n_j = \#(E^j)$.
We will denote by $\alpha_0$ the maximum percentage of outliers per node. 
That is, $\alpha_0$ is the maximum of  $\#(E_b^j)/n_j$ over all $j \in [n]$,
where throughout the rest of the paper $\#(\cdot)$ denotes the number of points in a set and $n_j$ is the degree of node $j$, $n_j = \#(E^j)$. 

The following notion of recovery threshold is related to the notion of a breakdown point in robust statistics. However, our goal is somewhat different, since we desire an exact estimator rather than an approximation, as is typically considered for classical breakdown points. This is similar to the notion of RSR breakdown point given in Section 1.1 of~\citet{maunu2019robust}.
\begin{definition}[Recovery Threshold]
	The recovery threshold of a robust rotation synchronization algorithm is the largest value of $\alpha_0$ such that the algorithm outputs an estimator $(\hat{\bR})$ that satisfies~\eqref{eq:reccond}.
\end{definition}

The simplest information-theoretic bound for the recovery threshold is $\alpha_0\leq 1/2$. Indeed, if $\alpha_0 > 1/2$, then an adversary could easily choose $E_b$ to have a subgraph that dominates $E_g$ with a consistent set of measurements for an alternative signal $(\bR^b) = (\bR_1^b, \dots, \bR_n^b)$. That is, the observations would be
\begin{equation}
    \bR_{jk} = \begin{cases}
    \bR_j^{\star \top} \bR_k^\star, & jk \in E_g \\
    \bR_j^{b \top} \bR_k^b, & jk \in E_b. \\
    \end{cases}
\end{equation}
If the adversary chooses the partition $E_g$ and $E_b$ properly, then one could easily think that $\bR^b$ is the true underlying signal. For our method, we obtain recovery thresholds for $\alpha_0$ that are smaller than $1/2$.

{
    On the other hand, the information theoretic bound may be much higher in special models. For example, suppose that $G([n],E)$ is an Erd\"os-R\'enyi graph with parameter $p$. Suppose further that each edge in $E$ is corrupted independently with probability $q$, and the corrupted measurements are i.i.d.~uniform on $\SO(D)$. \citet{wang2013exact} call this the uniform corruption model. Then,~\citet{singer2011angular} and~\citet{chen2016information} established the following information theoretic threshold for the rotation synchronization problem:
    \begin{equation}\label{eq:inforate}
        q= 1-\Omega\left(\sqrt{\frac{1}{p} \frac{\log(n)}{n}}\right).
    \end{equation}
    Notice that one needs $p \gtrsim \log(n)/n$ to ensure that the underlying Erd\"os-R\'enyi graph is connected. For fixed $p$, notice that one can take $q$ arbitrarily close to 1 (and so $\alpha_0$ is then very close to 1) as long as $n$ is sufficiently large. We later discuss how our main results extend to this model in Section~\ref{subsec:assumpdiscuss}, where we show that the DDS algorithm achieves optimal recovery rates with respect to $p$ (i.e., it can tolerate extremely sparse observation graphs).
}

\section{{An Adversarially Robust Algorithm for $\SO(2)$ Synchronization}}
\label{sec:tas}

To begin to build motivation for our method, we consider the case of synchronization over $\SO(2)$, where the method becomes considerably simpler due to its 1-dimensional manifold structure. First, Section~\ref{subsec:geomc1} gives definitions of some geometrical objects on $\SO(2)$, which we identify with $\C_1$ for mathematical convenience. 
Then, in Section~\ref{subsec:tas}, we define our $\SO(2)$ synchronization method, which we call Trimmed Averaging Synchronization (TAS) and is a special case of our later DDS algorithm. Finally, Section~\ref{subsec:tastheory} discusses the initialization and well-connectedness assumptions and uses these to give an adversarial recovery guarantee for the TAS algorithm.

\subsection{The Geometry of $\C_1$}
\label{subsec:geomc1}

We define a few structures related to the manifold $\C_1$.
The tangent space can be identified with $\R$. 
%and as we will see below, it is convenient to think of these as arguments of complex numbers. 
Let $v \in T_z \C_1$ be a unit direction in the tangent space at $z_j$ (i.e., $v = \pm 1$). 
The geodesic originating at $z_j$ in the direction $v$ is given by $\gamma(t) = e^{i v t} z_j, \ t \in [0, \pi/|v|]$.
The exponential map and inverse exponential map (logarithm map) on this 1-dimensional manifold are given by
\begin{equation}
    \Exp_z(\theta) = e^{i \theta} z, \theta \in (-\pi, \pi],\quad \Log_z(y) = \arg(y \overline{z}).
\end{equation}
Finally, the \emph{cut-locus} of a point $z \in \C_1$ is defined as the set of points for which there is not a unique geodesic from $z$. It is not hard to see that this is given by $ \mathsf{cut}(z) = \{-z\}$.

%As we will see later, multiple rotation averaging algorithms operate by iteratively computing averages of discrete datasets on $\C_1$. These discrete datasets are actually \emph{multisets}, since they may have repeated values. In the future, we will make clear when we refer to a multiset. 

Recall that we seek an underlying signal $\bz^\star \in \C_1^n$. Notice that its elements, $z_j^\star \in \C_1$ for $j \in [n]$, can be parameterized by angles, $z_j^\star = e^{i \theta_j^\star}$. This angle is also known as the argument of the complex number, and so we write $\arg(e^{i\theta}) = \theta$, where $\theta \in (-\pi, \pi]$. The angular, or geodesic, distance between $z_1$ and $z_2 \in \C_1$ is
\begin{equation}
\label{eq:ang_dist}
	d_\angle(z_1, z_2) = |\arg(z_1 \overline{z_2})|.
\end{equation}
For later reference, we plot the extended angular distance function in Figure~\ref{fig:energy}.

\begin{figure}[ht]
	\centering
	%{\LARGE Includegraphics isn't working...}
	\includegraphics[width=0.6\textwidth]{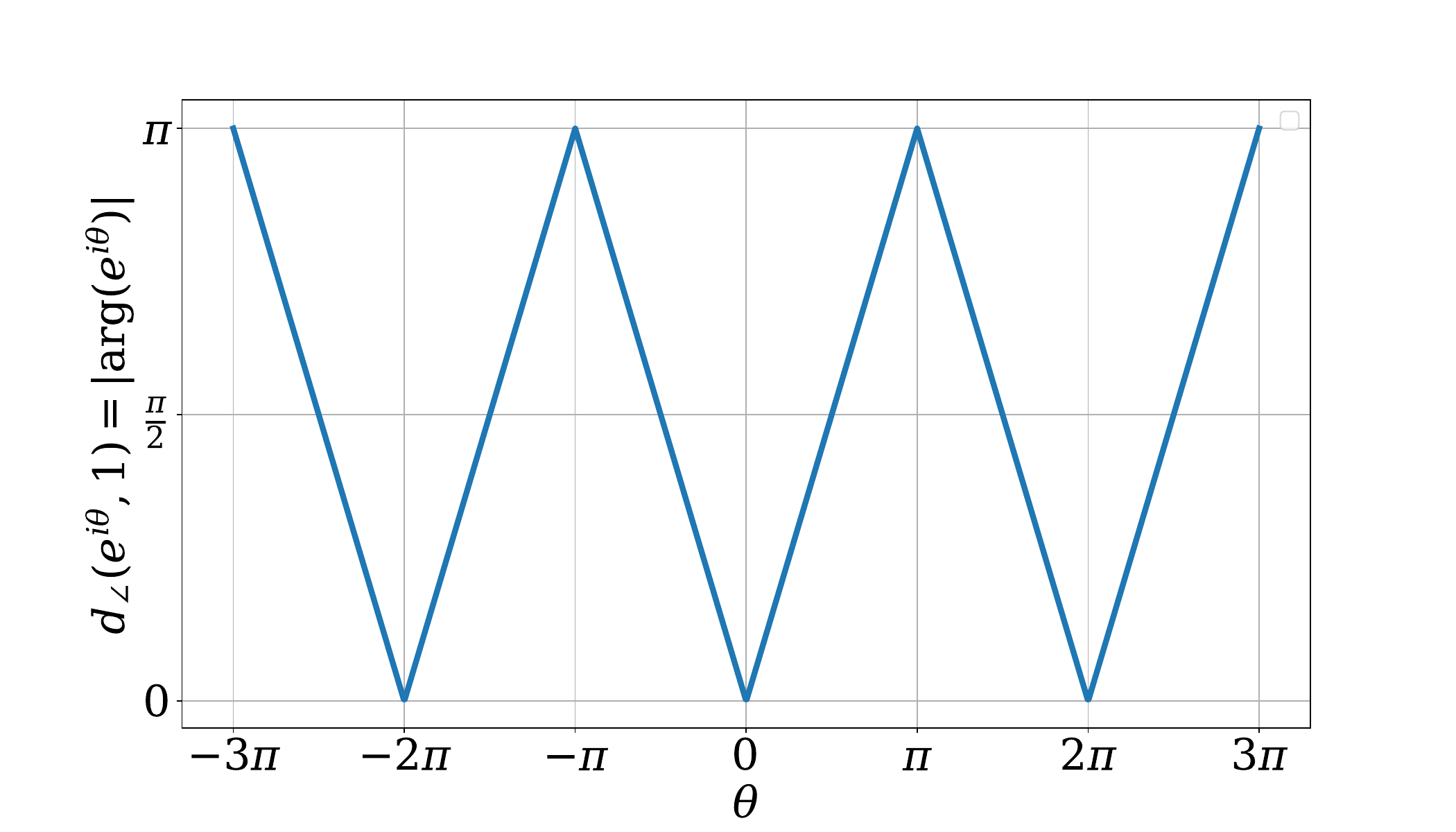}
	\caption{The angular distance function $d_\angle(e^{i \theta}, 1)$.}\label{fig:energy}
\end{figure}

Recall that if $jk \in E_g$, then the edge measurement is correct, that is, $z_{jk} = z_{jk}^\star$, where $z_{jk}^\star:=z_j^\star \overline{z_k^\star}$ is defined analogously to~\eqref{eq:meas}. 
For $jk \in E_b$, the measurement $z_{jk}$ is assumed to be an arbitrary element of $\C_1$.
{From the measurements $z_{jk}^\star$, $jk \in E_g$, $\bz^\star$ is only identified up to a global rotation, due to the ambiguity that $z_j^\star \overline{z_k^\star} = z_j^\star y \overline{y z_k^\star}$, $y \in \C_1$, and so $\bz^\star y$ generates the same pairwise measurements as $\bz^\star$.}

%Therefore, we define an equivalence relation on $\C_1^n$ by 
%\begin{equation}
%	\bz \sim \bz' \iff \bz = \bz' y, \text{ for some } y \in \C_1.
%\end{equation}
To deal with this ambiguity, the following function will be used to demonstrate convergence of a sequence to $\bz^\star$:
\begin{equation}
    \delta(\bz) = \max_{jk \in E} d_\angle(\overline{z_j^\star} z_j , \overline{z_k^\star}z_k ).%= \max_{jk \in E} \Big|\arg(z_j \overline{z_k} \overline{z_j^\star \overline{ z_k^\star}} )\Big|.
\end{equation}
This is again nothing but a function that measures the maximum distance between normalization products.
Notice that $\delta(\bz) = 0 \iff \bz = \bz^\star y$ for some rotation $y \in \C_1$. Therefore, convergence of $\delta(\bz)$ to zero indicates convergence of $\bz$ to $\bz^\star$, and an algorithm exactly recovers $\bz^\star$ iff $\delta(\bz) \to 0$.
%Finally, we note that in the case of $\SO(2)$, Assumption~\ref{assump:local} becomes the requirement that $ \delta(\bz(0)) < \pi$.

% {
% \subsection{Theoretical Assumptions}
% \label{subsec:assumptionso2}
% }

\subsection{Trimmed Averaging Synchronization}
\label{subsec:tas}

{
A natural way to solve the rotation synchronization problem involves energy minimization. The simplest strategy~\citep{govindu2001combining,martinec2007robust} attempts to minimize
\begin{equation}\label{eq:ls}
	\min_{\bz \in \C_1^n} F_{\angle}(\bz) := \sum_{jk \in E} d_\angle^2\left(z_j, z_{jk} z_k\right).
\end{equation}
A coordinate descent strategy to solve~\eqref{eq:ls} involves updating $z_j$ by solving
\begin{equation}\label{eq:mraupdate}
	\min_{z\in \C_1} \sum_{k \in E^j} d_\angle^2\left(z, z_{jk} z_k\right).
\end{equation}
Applying this sequentially over the indices $j=1,\dots,n$ results in the multiple rotation averaging (MRA) discussed in more detail in~\cite{hartley2013rotation}.
While such coordinate descent strategies generally lack coordination across all objects like global synchronization methods, they lead to algorithms that are more memory efficient and that can be decentralized easily.

}

Another simple way to robustify~\eqref{eq:mraupdate} is to select the average of all points that fall within a trimmed set, which results in a trimmed averaging procedure. To account for the 1-dimensional manifold structure of $\SO(2)$, we propose to do this trimming in the tangent space, which yields the TAS algorithm. An illustration of one trimmed averaging step is given in Figure~\ref{fig:tas}.

For a discrete $\mathcal{X} \subset \R$ and a fraction $0<p<1$, we write the $p$th quantile of $\mathcal{X}$ by $\mathcal{X}_{p}$.
It is convenient to define the trimming operator
\begin{equation}
	\cT_{\tau}\cX  = \bset{x \in \cX : \cX_{\tau} \leq x \leq \cX_{1-\tau}}.
\end{equation}
We also denote the average of a dataset $\cX \subset \R$ by 
$\ave(\cX)$. 
That is, $\ave(\cX) = \big(\sum_{x \in \cX} x\big)/\#(\cX)$.

%Therefore, in the case of $\SO(2)$, the method boils down to using trimmed averages~\citep{huber_book}.% to select a point in the depth set in our algorithm\footnote{A strategy based on Windsorized means would also work because the underlying principle would be the same.}. 
%This is the first instance of such methods being used for the synchronization problem.
Due to the simplified geometry of $\SO(2)$, we will show in the following that using this trimmed rotation averaging scheme converges to the underlying solution linearly when the percentage of outliers is at most $\alpha_0 < 1/4$ when $G$ is fully connected. In the case where $(G,[n])$ is not fully connected, the result is a corollary of our later Theorem~\ref{thm:sodrecovery} under a connectedness assumption on $(G,[n])$. 
We note that this fraction is similar to the one given in~\citet{lerman2019robust}, although there the bound is formulated for corrupted triangles in the graph. %Beyond exact recovery, trimmed means are also appealing due to the fact that they are more efficient than the median for Cauchy noise distributions~\citep{bloch1966note}, although it is unclear if such results translate to $\C_1$. 

% when $G$ is assumed to be fully connected, as well as an extension to the setting where $G$ is not fully connected.

%\end{equation}

%To heuristically motivate the algorithm, at an iteration $t$, the set $\{z_{jk} z_k(t) : k \in [n] \setminus j\}$ are estimates of the orientation $z_j$ based on the current guesses of $z_k(t)$. To update our guess of $z_j$, it makes sense to consider some form of average of these measurements. This is precisely the motivation behind multiple rotation averaging schemes. With this in mind, we must actually do the averaging over the manifold $\C_1$. In our procedure, instead of calculating the mean on the manifold $\C_1$, we instead find the mean in the tangent space and then project back to the manifold. Consider $\{\theta_{k}(t) =  \Log_{z_j(t)}(z_{jk} z_k(t)) : k \neq j\}$, which contains the angles of the estimates of $z_j$ with respect to $z_j(t)$. Our proposal is to take a trimmed average of these angles, and then use the exponential map to project this average back to $\C_1$. 
\begin{figure}[ht]
	\centering
	%{ \LARGE Includegraphics isn't working}
	\includegraphics[width=0.7\columnwidth,trim=0 150 0 110,clip]{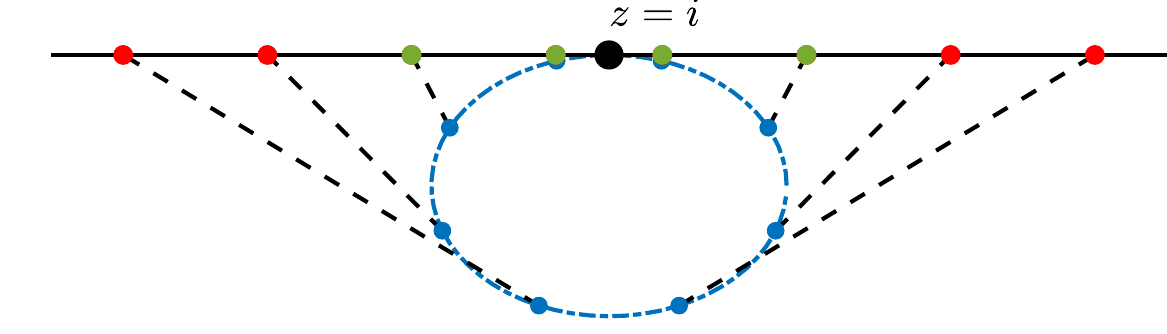}
	\caption{Illustration of the TAS algorithm at a fixed step and a fixed node $j$. The measurement is $z_j = z = i$. After projecting into the tangent space, the outermost points in red are filtered, and the green points are averaged. This trimmed average is then projected back to the manifold.}\label{fig:tas}
\end{figure}
 
For clarity, we give the TAS algorithm in Algorithm~\ref{alg:tas}. To allow for damping of the updates, we include the step-size parameter $\eta \in (0,1]$. When $\eta<1$, we refer to the algorithm as Damped TAS or DTAS for short.

% \begin{algorithm}[ht]
%     \SetAlgoLined
%     \DontPrintSemicolon
%     \caption{ $\eta$-Damped Trimmed Averaging Synchronization}\label{alg:tas}
%     \KwIn{$\bz(0)$, number of epochs $T$, damping parameter $\eta \in (0,1]$}
%     \For{$t = 1, \ldots, T$}{
%         $\bz(t+1) \gets \bz(t)$\;
%         $j = t \mod n$\;
%         $z_j(t+1) \gets \exp_{z_j(t+1)} \left[ \eta \cdot \ave \left( \cT_{0.25} \left\{  \log_{z_j(t+1)} \left(z_{jk} z_k(t) \right) : k \neq j \right\} \right) \right]$\;
%     }
%     \Return{$\bz(T)$}\;
%     %   \textbf{return } $\Sigma_{T}$;
% \end{algorithm}

\begin{algorithm}[ht]
    \caption{ $\eta$-Damped Trimmed Averaging Synchronization}\label{alg:tas}
    \begin{algorithmic}
    \REQUIRE{$\bz(0)$, number of iterations $T$, damping parameter $\eta \in (0,1]$, trimming parameter $\tau$}
    \FOR{$t = 1, \ldots, T$}
        \STATE{$j = t \mod n$}
         \STATE{$z_j(t+1) \gets \Exp_{z_j(t+1)} \left[ \eta \cdot \ave \left( \cT_{\tau} \left\{  \Log_{z_j(t+1)} \left(z_{jk} z_k(t) \right) : k \in E^j \right\} \right) \right]$}
         \STATE{$z_k(t+1) \gets z_k(t), \ k \neq j$}
     \ENDFOR
    \RETURN{$\bz(T)$}
    \end{algorithmic}
    %   \textbf{return } $\Sigma_{T}$;
\end{algorithm}

%CONTINUE

\subsection{Recovery Guarantees for DTAS}
\label{subsec:tastheory}

{ We begin by discussing the assumptions that will make a synchronization problem tractable for TAS.
The first assumption we require is a good initialization, which is common in the analysis of such nonconvex methods.
\begin{assumption}\label{assump:initso2}
	The initial set of rotations $\bz(0) \in \C_1^n$ lies within a $\pi/2$-neighborhood of $\bz^\star$: that is, there exists a $w \in \C_1$ such that 
	\begin{equation}
	    d_\angle(\overline{z_j^\star} z_j, w) < \pi/2, \ j=1, \dots, n.
	\end{equation}
	Note that this is equivalent to the assumption that $\delta(\bz) < \pi$.
\end{assumption}
}

While corruptions are arbitrary, we require an assumption on the underlying graph $(G,[n])$. It essentially requires that the graph is sufficiently well connected.
{ 
\begin{assumption}[$\zeta$-Well-connectedness condition]\label{assump:weakwellconnect}
    For a fixed $\zeta \in (0,1]$, for any $J \subset [n]$ such that $\#(J) \leq n/2$, there exists an index $j \in J$ such that 
    \begin{equation} \label{eq:weakwellconnect}
    	\Big(\frac{2}{\zeta} - 1\Big) \#\Big[ E^j \cap \big([n] \setminus J\big) \Big]  > \#\Big[ E^j \cap J \Big]. 
    \end{equation}
    %In other words, the node $j$ is connected to $(1-\zeta) n_j$ more nodes inside $E^j$ than outside of $J$ than inside of $J$.
\end{assumption}
In words, this assumption requires that inside any set of at most $n/2$ nodes, there is a node that is connected to a significant number of nodes outside this set. The condition in this assumption is equivalent to requiring that
\begin{equation}\label{eq:weakwellconnect2}
     \#\Big[ E^j \cap J \Big] < \Big(1-\frac{\zeta}{2}\Big) n_j. 
\end{equation}
We include a discussion of this condition and its connection with random graphs, conductance, and expanders later in Section~\ref{subsec:assumpdiscuss}.

For the case of Assumption~\ref{assump:weakwellconnect} with $\zeta = 1$, we just call the graph well-connected. While fully connected graphs satisfy this condition, there exist many more examples of graphs satisfying it with $\zeta = 1$ as well, and we give some examples of some simple graphs that meet this assumption in Figure~\ref{fig:well}.

}

\begin{figure}[ht]
	\centering
	%{ \LARGE Includegraphics isn't working}
	\includegraphics[width=0.8\columnwidth]{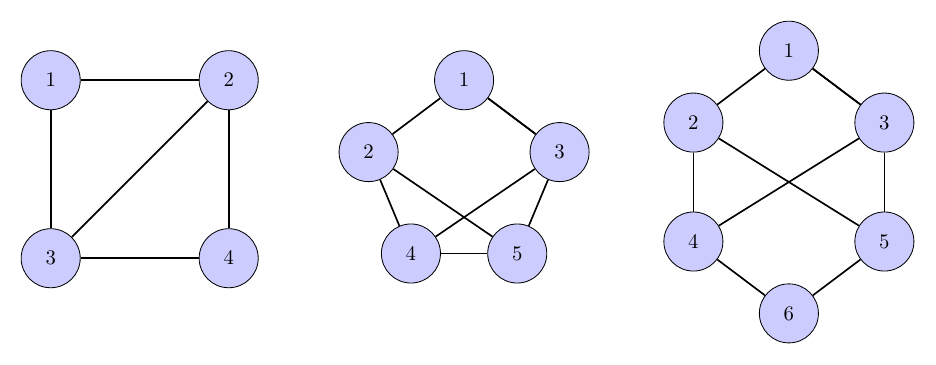}
	\caption{Examples of graphs that satisfy the well-connectedness condition for $n=4,5$ and $6$. In each of these graphs,~\eqref{eq:weakwellconnect} is satisfied with $\zeta=1$: that is, all subsets $J$ of size at most $n/2$, there exists a node $j \in J$ such that $\#(E^j  \cap ([n] \setminus J)) > \#(E^j  \cap J)$.}\label{fig:well}
\end{figure}

The following theorem gives the main recovery result for the DTAS algorithm. While the algorithm converges linearly, the rate we derive depends on $n$ and is worst-case. In the few simulations we have run, the algorithm seems to converge at a faster rate that merits more study. Also, a more complicated proof may yield linear convergence in the general case of well-connected $G$, but for sake of brevity, we only prove it for the fully connected case.
\begin{theorem}\label{thm:tasrecovery}
	Suppose that $\alpha_0 < \zeta/4$, Assumption~\ref{assump:initso2} holds, $G$ satisfies Assumption~\ref{assump:weakwellconnect} with parameter $\zeta$, and $[\bz(t)]_{t \in \nats}$ is the sequence generated by DTAS, for $\eta \in (0,1)$ and $\tau = \zeta/4$. Then, $\delta(\bz(t)) \to 0$, and the algorithm exactly recovers $\bz^\star$.
	Furthermore, in the case where $G$ is fully connected, the DTAS algorithm linearly converges to $\bz^\star$.
\end{theorem}
    \begin{proof}
    The proof of convergence under well-connectedness follows from the fact that,  when updating index $j$ at iteration $t$, the selection rule defined by choosing the trimmed average yields a point in the interior of $\cT_{\tau}(\{  \Log_{z_j(t)}(z_{jk} z_k(t)) : k \in E^j\})$. The proof then follows from the proof of Theorem~\ref{thm:sodrecovery}.
	 
	 To see linear convergence in the fully connected case, we prove that all normalization products $\overline{z_j^{\star}} z_j(t)$ contract during each pass over the dataset. Denote
	\begin{equation}
		\delta_j = \theta_j(t) - \theta_j^\star \in [-\delta(\bz(t))/2, \delta(\bz(t))/2].
	\end{equation}
	These are the translation of the normalization products to the angular coordinates of the points $z_1(t), \dots, z_n(t)$. 
	Also, define the sets
	\begin{equation}
		I_+(t) = \bset{k : \arg(\overline{z_j^\star} z_j(t) ) > 0}, \quad I_-(t) = \bset{k : \arg(\overline{z_j^\star} z_j(t) ) \leq 0}.
	\end{equation}	
	In this proof, we will write $\delta = \delta(\bz(t))$ as a shorthand.
	Notice that we must have 
	$$\min(\# I_+(t), I_-(t)) \leq n/2,$$
	unless $I_+(t) = I_-(t)) = [n]$, in which case $\bz(t)= \bz^\star w$ for some $w \in \C_1$ and  $\bz(t)$ recovers $\bz^\star$.

	For the update with respect to index $j$, all good pairwise measurements must lie in $-\delta_j + [-\delta/2, \delta/2]$. Since there are at least $3n_j/4$ good measurements, all trimmed points must lie in this interval as well. Therefore, for all $j \in I_-(t)$, after updating we have
	\begin{equation}
		\overline{z_j^\star} z_j(t) \in \exp\big[ i [-\delta/2, \eta \delta/2) \big].
	\end{equation}
	
	Using this fact, we will now show that the indices in $I_+(t)$ must move inwards. Indeed, since $\# I_+ (t) \leq n/2$, we must have $\# (E_g^j \cap I_-(t)) \geq 1$ for all $j \in I_+(t)$. Therefore, for each trimmed mean for $j \in I_+(t)$, we have
	\begin{align} 
		\ave \left( \cT_{0.25} \left( \bset{  \Log_{z_j(t)}(z_{jk} z_k(t)) : k \neq j} \right) \right) &\leq \frac{2}{n-1} \Big[ \eta \frac{\delta}{2}-\delta_j + \Big( \frac{n-1}{2} - 1 \Big) \Big( \frac{\delta}{2} -\delta_j \Big) \Big] \\ \nonumber
													     & \leq \Big( \frac{n-3}{n-1} - \frac{2}{n-1}\eta \Big) \cdot \frac{\delta}{2} - \delta_j.
	 \end{align}
	 Thus,
	\begin{align} 
		\eta \ave \left( \cT_{0.25} \left( \bset{  \Log_{z_j(t)}(z_{jk} z_k(t)) : k \in E^j} \right) \right) + \delta_j &\leq \eta \Big( \frac{n-3 - 2 \eta }{n-1}\Big) \cdot \frac{\delta}{2} + (1-\eta) \delta_j \\ \nonumber
															     &\leq \eta \Big( \frac{n-3}{n-1}\Big) \cdot \frac{\delta}{2} + (1-\eta) \frac{\delta}{2} \\ \nonumber
															     &= \frac{\delta}{2} \Big( 1 - \eta \Big[\frac{2}{n-1} \Big] \Big).
	 \end{align}
	
	 After the coordinate update, we have that for all $j \in I_+(t)$,
	\begin{equation}
		\overline{z_j^\star} z_j(t) \in \exp\Big( i \Big[-\frac{\delta}{2}, \Big(1 - \eta \Big( \frac{2}{n-1} \Big)  \frac{\delta}{2} \Big) \Big] \Big).
	\end{equation}

	After repeating this argument for all $j$ over the course of an epoch, or pass over all indices $j=1,\dots, n$, this yields that
    $$\overline{z_j^\star}  z_j(t+1) \in \exp\Big( i \Big[-\frac{\delta}{2}, \Big(1 - \eta \Big( \frac{2}{n-1} \Big)  \frac{\delta}{2} \Big) \Big] \Big),$$
	    as long as $\eta < (n-1)/(n+1)$.
	    The width of this interval is $ (n - 1 - \eta) \delta(\bz(t))/(n-1)$, which yields the desired result.\qed
\end{proof}

\section{Robust Synchronization over $\SO(D)$}
\label{sec:dds}

This section presents a novel algorithm for robust synchronization over the rotation group, $\SO(D)$. We assume a fixed observation graph $G$ that encodes which pairwise rotations we observe. The pairwise rotations are written as ${\bR}_{jk} \in \SO(D)$, where the good edges $jk \in E_g$ have the associated observation $\bR_j^\star \bR_k^{\star \top}$, and the bad edges are arbitrarily chosen from $G$ and have arbitrarily corrupted measurements.

To proceed, we must make clear our goal for the synchronization problem, since there is a well known ambiguity, similar to the one encountered for $\SO(2)$ synchronization in Section~\ref{subsec:geomc1} -- we can only recover $\bR^\star$ up to right multiplication by an element of $\SO(D)$. This is because, after this multiplication, one arrives at the same pairwise measurements in~\eqref{eq:meas}. This is a form of rotational symmetry in the nonconvex problem, which may be leveraged to develop tractable nonconvex programs~\citep{zhang2020symmetry}. Exactly recovering the ground truth measurements $(\bR^\star) = (\bR_1^\star, \dots, \bR_n^\star) \in \SO(D)^n$ up to right multiplication by $\bS \in \SO(D)$ is equivalent to finding a set of rotations $(\bR) = (\bR_1, \dots, \bR_n)$ such that
\begin{equation}\label{eq:reccond}
    \bR_1^{\star \top} \bR_1 = \dots =  \bR_n^{\star \top} \bR_n = \bS,
\end{equation}
for some $\bS \in \SO(D)$.
We refer to the set of rotations $\bR_j^{\star \top} \bR_j$, $j=1,\dots,n$, as \emph{normalization products} since, when $(\bR) = (\bR^{\star} \bS)$, they reveal the normalization factor that multiplies each element of $(\bR^\star)$ from the right. One could extend this discussion to the case of approximate recovery by requiring that the normalization products are approximately equal.

With our goal now in mind, we begin in Section \ref{subsec:sod} by discussing the geometry of the manifold $\SO(D)$ and presenting some basic geometric results that will be used in our main theorem.
%Then, in Section \ref{subsec:robnotions}, we outline some notions of robustness for the synchronization problem, as well as our main assumptions. 
Following this, Section \ref{subsec:depth} reviews the concept of halfspace depth from robust statistics, which will be the core tool that we use to construct our algorithm. In Section \ref{subsec:ddsalg}, we give outline the DDS algorithm. Then, in Section~\ref{subsec:ddsrecov}, we give the theoretical guarantees that constitute the main innovations of this work. We finish in Section \ref{subsec:assumpdiscuss} with a discussion of the assumptions we make in our theorem.

\subsection{The Manifold Structure of $\SO(D)$}
\label{subsec:sod}

The rotation synchronization problem is obviously a robust recovery problem on the product Riemannian manifold $\SO(D)^n$. Therefore, in the following, we freely use concepts from Riemannian geometry, and specifically those concepts related to the geometry of $\SO(D)$. %For a brief review of some of these concepts, the reader should consult Appendix \ref{app:riem}. 

\subsubsection{{Riemannian Geometry of $\SO(D)$}}

The set of rotations $\SO(D)$ is a $D(D-1)/2$-dimensional Lie group that has a natural Riemannian structure. The bi-invariant distance metric $d: \SO(D) \times \SO(D) \to [0,\lfloor \frac{D}{2} \rfloor \pi]$ is given by
\begin{equation}
	d(\bR_1, \bR_2) = \|\log(\bR_1 \bR_2^\top)\|_F,
\end{equation}
where $\log$ is the matrix logarithm. The corresponding Lie algebra is $\mathfrak{so}(D)$, the set of $D \times D$ skew-symmetric matrices. The tangent space of $\SO(D)$ at $\bR \in \SO(D)$ is 
$$ T_{\bR} \SO(D) = \{\bDelta \in \mathbb{R}^{d \times d} : \bR^\top \bDelta \in \mathfrak{so}(D)\} .$$
Notice that every tangent vector $v \in T_{\bR} \SO(D)$ has a corresponding element of $\mathfrak{so}(D)$, which we denote by $v_{\mathfrak{so}(D)}$.
The corresponding Riemannian metric (which is an inner product and thus should not be confused with a distance metric) for $v, w \in T_{\bR} \SO(D)$ is given by $\langle v,w \rangle_{\bR} = \Tr(v^\top w)/2 = \Tr(v_{\mathfrak{so}(D)}^\top w_{\mathfrak{so}(D)})/2$. Equipped with this metric, $\SO(D)$ is a Riemannian manifold with nonnegative sectional curvature. 
An open ball with respect to the metric $d$ is written as $B(\bR, r)$, where the radius is $r$ and the center is $\bR$. Its closure and boundary are $\overline{B(\bR, r)}$ and $\partial B(\bR, r)$, respectively.

The exponential map is given by
\begin{equation}
	\Exp_{\bR}: T_{\bR}\SO(D) \to \SO(D), \quad \Exp_{\bR}(\bU) = \bR \exp(\bR^\top \bU),
\end{equation}
where $\exp$ is the matrix exponential. %#The exponential map $\Exp_{\bR}$ on $\SO(D)$ is a bijection from a subset of $T_{\bR} \SO(D)$ to $B(\bR, \lfloor \frac{D}{2} \rfloor\pi)$.
The logarithmic map is the inverse of this: 
\begin{equation}
	\Log_{\bR}: \SO(D) \to T_{\bR}\SO(D), \quad \Log_{\bR}(\bS) = \bR \log (\bR^\top \bS).
\end{equation} 
The geodesic between $\bR, \bS \in \SO(D)$ is written as $\overrightarrow{\bR\bS}(t) = \Exp_{\bR}(t \Log_{\bR}(\bS)) $, for $t \in [0,1]$.
   In the following, we use the notation for a halfspace of $T_{\bR} \SO(D)$,
\begin{equation}\label{eq:halfsp}
	\cH(\bR, v) = \{u \in T_{\bR} \SO(D) : \langle u, v \rangle_{\bR} >0, \ \Exp_{\bR}(tu) \text{is a geodesic for } t \in [0,1] \}.
\end{equation}

\subsubsection{{Local Convexity Properties of $\SO(D)$}}

We continue by recalling some local convexity properties of manifolds like $\SO(D)$.
The following result on the convexity of sufficiently small balls is standard in the literature~\citep{karcher1977riemannian,afsari2009means,petersen2016riemannian}.
\begin{theorem}\label{thm:intgeo}[Convexity of small balls]
	In a closed ball $\overline{B(\bC, r)} \subset \SO(D)$ with $r < \pi/2$, the squared distance metric $d^2$ is strictly convex. This implies, in particular, that for all $\bR_0, \bR_1 \in \overline{B(\bC, r)}$, $\overrightarrow{\bR_0\bR_1}(t) \in B(\bC, r)$ for all $t \in (0,1)$. %is strictly convex for all $r < \pi/2$.  That is, for all $\bR_0, \bR_1 \in \overline{B(\bC, r)}$, $\overrightarrow{\bR_0\bR_1}(t) \in B(\bC, r)$ for all $t \in (0,1)$.
\end{theorem}

Note that the closed ball in the previous theorem has the property that the interior of any nonconstant geodesic lies strictly in the interior of the ball.
The following result is also readily apparent. It states that, for boundary points on a sufficiently small ball in $\SO(D)$, all interior directions are contained in a halfspace.
\begin{corollary}\label{cor:sodhalfsp}
    Let $\overline{B}$ be a ball on $\SO(D)$ with radius $r < \pi/2$ and $\bR \in \partial B$. Then, there is a halfspace $\cH \subset T_{\bR}\SO(D)$ such that $\Log_{\bR} \overline{B} \subset \cH$.
\end{corollary}
% \begin{proof}
%     We have, for any geodesic $\gamma(t)$ with $\dot{\gamma}(0) = v$, 
%     \[
%         \partial_t d^2(\gamma(t), \bC)\Big|_{t=0} = -\langle v, \Log_{\gamma(0)} (\bC) \rangle_{\gamma(0)}.
%     \]
%     This is strictly negative as long as $\angle(v, \Log_{\gamma(0)} (\bC) ) < \pi/2$, i.e., for all directions defined by the halfspace $H(\gamma(0),\Log_{\gamma(0)} (\bC))$. This means that $H(\gamma(0),\Log_{\gamma(0)} (\bC))$ contains all directions that point into the interior of the ball.
%     Due to the fact that the distance function is strictly convex along non-radial directions~\cite{afsari2009means}, we obtain the desired result.
% \end{proof}

An important component of our later theoretical results relies on showing that the radius of the smallest ball containing a set of rotations shrinks.
The final lemma of this section states that if a discrete set of rotations $\cR$ is contained in a ball, and if half of the ball contains no boundary measurements, then the set $\cR$ is actually contained in a ball of smaller radius. Thus, for such a set of rotations, this gives us a sufficient condition for decreasing the radius of the smallest containing ball.
\begin{lemma}\label{lemma:shrinkball}
    Let $B=B(\bC,r) \subset \SO(D)$ and $\cR \subset \overline{B}$ a finite set of rotations. Suppose that there exists $v \in T_{\bC} \SO(D)$ such that 
    $$ \Exp_{\bC}[\overline{\cH(\bC,-v)}] \cap \partial B \cap \cR  = \emptyset.$$
    Then, $\cR$ is contained in a ball with radius less than $r$.
\end{lemma}
\begin{proof}
    Let $c(t)$ be the geodesic $\Exp_{\bC}(tv)$. We claim that, for $t$ sufficiently small, $\cR \subset B(c(t),r)$, which is an open ball. Heuristically, one should expect this to be true, since the closed halfspace $\overline{\cH(\bC,-v)}$ contains no boundary points, and so moving the center a small amount in the $v$ direction keeps all points within the ball.
    
    By the first order approximation to $d^2(c(t), \bR)$ and since $\angle( v, \Log_{\bC} \bR) < \pi/2$, we have
    \begin{equation}\label{eq:dbd1}
        d(c(t),\bR)^2 < d(c(0),\bR)^2, \ \forall \bR \in \cR \cap \Exp_{\bC}(\cH(\bC,v)),
    \end{equation}
    for $t$ sufficiently small.
    On the other hand, since 
    $$\Exp_{\bC}[\overline{\cH(\bC,-v)} ] \cap \partial B \cap \cR \subset B(\bC,r),$$ 
    the distance to $\bC$ over all $\bR \in \Exp_{\bC}[\overline{\cH(\bC,-v)} ] \cap \partial B \cap \cR$ is bounded away from $r$. By continuity of $d(c(t),\bS)$ for all $\bR \in \cR$, this implies that there is an $\epsilon$ such that
    \begin{equation}\label{eq:dbd2}
        d(c(t),\bR) < r, \ \forall \ t \in (0,\epsilon).
    \end{equation}
    Putting~\eqref{eq:dbd1} and~\eqref{eq:dbd2} together implies that a small shift of the ball results in a new center such that $\max_{\bR \in \cR} d(c(t),\bR) < r$. In turn, this means that all points of $\cR$ lie in a ball $\overline{B(c(t),r')}$ with $r' < r$. \qed
\end{proof}

\subsection{Tukey Depth and its Properties}
\label{subsec:depth}

We will use the concept of \emph{Tukey depth} to determine descent directions on the manifold $\SO(D)^n$, although other notions of depth could potentially be used as well (see Ch. 58 of~\citet{toth2017handbook} for a discussion of different notions of depth). In Euclidean space, the Tukey depth of a point $\bx\in \R^D$ in a dataset $\cX = \{\bx_1, \dots, \bx_n\} \subset \R^D$ is given by
\begin{equation}\label{eq:depth}
    \depth(\bx, \cX) = \min_{\bu \in S^{D-1}} \#\{\bx_i \in \cX: \bu^\top (\bx_i - \bx) \geq 0\}
\end{equation}
The depth is therefore the minimum number of points contained in any halfspace that has $\bx$ in its separating hyperplane.
A natural robust estimator is then the point of maximum depth, which is also called the Tukey median. 
%This point is also called the Tukey median, which is mathematically defined as
%\begin{equation}\label{eq:tukeymedrd}
 %   \argmax_{\bx} \depth(\bx, \cX).
%\end{equation}
%The Tukey median is not in general unique~\citep{donoho1992breakdown}. However, since all Tukey medians must lie in the intersection of halfspaces, they form a convex set, and so it is usually convenient to identity the Tukey median with the center of mass defined by the set~\eqref{eq:tukeymedrd}. The Tukey median is computationally intensive to calculate in high-dimensions CITE, although efficient algorithms exist for $d \leq 3$ CITE. In any case, in the pursuit of finding guarantees, we consider such estimators, and we leave the discussion of more practical estimators to future work.
The $\beta$-depth level set for $\beta \in [0,1]$ is defined by
\begin{equation}
    \cD_{\beta}(\cX) = \{\by \in \R^D : \depth(\by, \cX) \geq \beta\#(\cX) \}.
\end{equation}
This level set is convex and compact, and its boundary is made up of hyperplanes defined by sets of $D$ points~\citep{liu2019fast}.
This function will be used in the construction of our algorithm.

{
As an example, consider the 1-dimensional dataset $\cX=\{x_i\}$. Here, the formulation of depth is quite simple:
\begin{equation}
    \depth(x, \cX) = \min(\#\{x_i \leq x\}, \#\{x_i \geq x\}) .
\end{equation}
With this in mind, the $\beta$-depth level set is $ \cD_{\beta}(\cX) = [x_{(\lceil \beta n \rceil)}, x_{(\lfloor(1-\beta) n \rfloor)}],$ where $x_{(i)}$ denotes the $i$th order statistic. The Tukey median in this case is just the median.
}

We recall the following theorem, which bounds the maximum possible depth within a general dataset. Notice that, in particular, this guarantees that the depth level set $\cD_{\beta}(\cX)$ is nonempty for all $\beta \leq 1/(D+1) $.
% A set of points lies in \emph{general position} if no subset of $k$ points lies in a $k-2$ dimensional affine subspace for $k=2,\dots,D+1$.
% \begin{proposition}[\cite{donoho1992breakdown}]
% 	Suppose that $\cX$ is a set of $n$ points in general position. Then, the point of maximum depth in $\mathbb{R}^D$ with respect to $\cX$ has depth between $\lceil n/(D+1) \rceil$ and $\lceil n/2 \rceil$.
% \end{proposition}
\begin{proposition}[\cite{rado1946theorem}]\label{prop:mindepth}
	Suppose that $\cX$ is a set of $n$ points. Then, the maximum depth in $\cX$ is bounded below by $\lceil n/(D+1) \rceil$.
\end{proposition}

A particularly useful property of Tukey depth is that it is \emph{affine equivariant}, that is, it is stable under affine transformations.
\begin{lemma}[\cite{donoho1992breakdown}]
    $$ \depth(\bA \bx+\bb, \bA \cX + \bb) = \depth(\bx, \cX).$$
\end{lemma}
This implies, in particular, that things behave nicely if we change the inner product on $\mathbb{R}^D$.
\begin{align}
     \min_{\bu \in S^{D-1}} \#\{\bx_i \in \cX: \bu^\top \bA (\bx_i -\bx) \geq 0\} &= \depth(\bA \bx, \bA \cX) \\ \nonumber
    &= \depth(\bx, \cX).
\end{align}
With this property, we can map between the depth regions in $(\mathbb{R}^D, \langle\cdot, \cdot \rangle)$ and $(\mathbb{R}^D, \langle\cdot, A\cdot \rangle)$ by
\begin{equation}
    A \cD_{\beta}(\cX) =  \cD_{\beta}(A\cX).
\end{equation}
This affine equivariance implies, in particular, that Proposition~\ref{prop:mindepth} extends to datasets in tangent spaces of manifolds.
%This property is quite useful on a Riemannian manifold.

We finish with a simple lemma on depth level sets that will be the key to the robustness guarantee for our later algorithm. {In simple words, this lemma guarantees sufficient conditions for a depth level set to contain a nonzero value in a dataset containing many zeros.} In the following, datasets are represented by multisets and may contain duplicate points. For a halfspace $H(\bzero, \bv):= \{\bx \in \R^D : \bv^\top \bx \geq 0\} \subset \R^D$, we write its separating hyperplane as $L(\bzero,\bv)$.

{
% A direct extension of this lemma yields the following.
\begin{lemma}\label{lemma:depthcont}
   Suppose that we have a dataset $\cX = \{\bx_1, \dots, \bx_n\} \in \R^D$ and a subset $\cY \subset \cX$ that satisfies the properties i) $\#(\cY) > n - n(\zeta /(2D+2))$, ii) There exists closed halfspace $\overline{H(\bzero,\bv)} \subset \R^D$ such that $\overline{H(\bzero,\bv)} \supset \cY$, and iii) the only points of $\cY$ in $L(\bzero, \bv)$ are $\bzero$.
   Then, $\cD_{\zeta/(2D+2)}(\cX) \subset \mathrm{conv}(\cY) \subset (H(\bzero, \bv) \cup \{\bzero\})$. Beyond this, if $\#(\cY \cap L(\bzero, \bv)) < (1-\zeta/2)n$, then $(\cD_{\zeta/(2D+2)}(\cX)) \cap H(\bzero, \bv) \neq \emptyset$.
\end{lemma}

\begin{proof}

%Suppose that we have a dataset $\cX = \{\bx_1, \dots, \bx_n\} \in \R^D$, and further that there is a halfspace $H(\bzero,\bv)$ that contains less than $\beta n$ points. Then, for all $\by \in H(\bzero, \bv)$, $\depth(\by, \cX) < \beta$. Take $\beta = 1/(2D+2)$. 
It is obvious by the properties of depth that $\cD_{\zeta /(2D+2)}(\cX) \subset \mathrm{conv}(\cY) \subset (H(\bzero, \bv) \cup \{\bzero\})$, since any point on the boundary of $\mathrm{conv}(\cY)$ has depth less than $\zeta /(2D+2)$.
By Proposition~\ref{prop:mindepth}, there is a point of depth at least $\zeta/(2D+2)$, and so $D_{\zeta/(2D+2)}(\cX)$ is nonempty. 

Suppose that less than $(1-\zeta/2)n$ points in $\cY$ are zero, that is, $\#(\cY \cap L(\bzero, \bv)) <(1-\zeta/2)n$. We claim that $D_{\zeta/(2D+2)}(\cX) \cap H(\bzero, \bv) \neq \emptyset$. Define the auxiliary set $\cZ = \cX \setminus (\cY \cap L(\bzero, \bv))$, that is, $\cZ$ removes the $\bzero$ values in $\cY$ from $\cX$. Since $\#(\cY \cap L(\bzero, \bv)) < (1-\zeta/2)n$, we have that $m = \#(\cZ) \geq (\zeta/2)n$. Within $\cZ$, there is a point $\hat{z}$ of depth at least $m/(D+1)$. Further, since $n \leq 2m/\zeta$ and $\#(\cX \cap H(\bzero, -\bv)) < n\zeta/(2D+2)$, we have
\begin{align*}
    \#(H(\bzero, \bv) \cap \cZ ) &> m - \frac{n \zeta}{2D+2} \\
    &\geq m - \frac{m}{D+1)}\\ 
    &= m\frac{D}{D+1},
\end{align*}
which implies that $\hat{z}$ must lie in $H(\bzero, \bv)$.
On the other hand, since $m \geq \zeta n/2$, we have that 
\[
    \depth(\hat{\bz}, \cX) \geq \frac{m}{D+1} \geq \frac{\zeta n}{2(D+1)},
\]
which means that $\hat{z}$ is a point of depth at least $\zeta /(2D+2)$ in $\cX$. \qed
\end{proof}
}

\subsection{Depth Descent Synchronization}
\label{subsec:ddsalg}

We now use the results of the previous sections to derive the DDS algorithm.  %Over the iterations, we update the rotations sequentially, and so at time $t$ we update index $j = t \mod n$.

\subsubsection{{The General DDS Algorithm}}

We assume a \emph{selection rule} $\cS_{\bR}$ on convex, compact subsets of $T_{\bR} \SO(D)$ for all $\bR\in \SO(D)$. In particular, for our theorem, we assume that this selection rule chooses a nonzero point from the convex set if possible, and otherwise outputs zero. To select this point, one selection rule could be to take a point uniformly at random, or another could be to take the center of mass. As our theorem makes clear, the choice of this selection rule does not affect the exact recovery result, but it may change convergence rates. Since we do not give quantitative convergence rates in this work, we leave this choice as arbitrary. 

In any case, given a selection rule $\cS_{\bR}$ over such subsets of $T_{\bR} \SO(d)$, suppose our estimated rotations at time $t$ are $(\bR(t))$. Our update direction at this time for index $j = t \mod n$ is defined by
\begin{align}\label{eq:vj}
        \bv_j(t) &=  \cS_{\bR}( \cD_{\beta} (\{\Log_{\bR_j(t)} \bR_{jk}\bR_k(t): k \in E^j\})).
\end{align}
{In words, the direction $\bv_j(t)$ is a direction in the tangent space at $\bR_j(t)$ that is sufficiently deep with respect to the neighbor measurements $\{\Log_{\bR_j(t)} \bR_{jk}\bR_k(t): k \in E^j\}$.}
%It therefore suffices to define $\cS$ only on compact, convex sets in a Hilbert space. Two examples of such a selection rule would be to sample from the convex set uniformly at random, or to take its center of mass.
%With  direction  $\bv_j(t)$ is a point somewhere within the depth set of level $\beta$. For example, the algorithm can select a point uniformly at random from this set. 
%Define the set 
%\[
%    I_j(t) = \{k : \bR_j(t) = \bR_{jk} \bR_k(t) \}
%S_j(t) = \bset{\log_{\bR_j(t)} (\bR_{j1} \bR_1(t)), \dots, \log_{\bR_j(t)} (\bR_{jn} \bR_n(t)) : \bR_{jk} \bR_k(t) \neq \bR_j(t)} 
%.\] 
% \begin{equation}
% \bv_j(t) =  \argmax_{v \in T_{\bR_j(t)}} \depth\big(v, \{\log_{\bR_j(t)} \bR_{jk} \bR_k(t) : k \in E^j \setminus I_j(t) \}\big) 
% \end{equation}
Given $\bv_j(t)$, the algorithm updates
\begin{equation}\label{eq:depthseq}
	j = t \mod n, \ \bR_j(t+1) = \Exp_{\bR_{j}(t)} (\eta(t) \bv_j(t)), \ \bR_k(t+1) = \bR_k(t) \text{ for } k \in [n] \setminus \{j\}.
% 	\begin{cases}
% 		\exp_{\bR_j(t+1)} (\eta(t) \bv_j(t)), \#(I_j(t)) \leq n/2, \\
% 		\bR_j(t), \#(I_j(t)) > n/2.	
% 		\end{cases}.
\end{equation}
for a chosen step size $\eta(t) \in (0,\eta^\star(D)]$. Our theory below restricts this step size according to Theorem 4.2 of~\cite{afsari2013convergence}: in the case of $D=2$ or $3$, one can take $\eta^\star(D) = 1$, while for $D>3$ the upper bound is more restrictive (the reader can consult the discussion in \cite{afsari2013convergence} for a more thorough discussion of the bound).
%It is essential that this set always exists for this map to be closed.
For sake of clarity, we write the full DDS algorithm in Algorithm~\ref{alg:dds}.
% \begin{algorithm}[ht]
%     \SetAlgoLined
%     \DontPrintSemicolon
%     \caption{ Depth Descent Synchronization for $\SO(3)$}\label{alg:dds}
%     \KwIn{$\bR(0)$, number of epochs $T$, selection rule $\cS$, $\eta \in (0,1]$ step size, $\beta$ depth parameter}
%     \For{$t = 1, \ldots, T$}{
%         $j = t \mod n$\;
%          $\bv_j(t) =  \cS_{\bR}( D_{\beta} (\{\log_{\bR_j(t)} \bR_{jk}\bR_k(t): k \in E^j\}))$\;
%         $\bR_j(t+1) \gets \Exp_{\bR_{j}(t)} ( \eta \bv_j(t))$\;
%         $\bR_k(t+1) \gets \bR_k(t), k \neq j$\;
%     }
%     \Return{$\bR(T)$}\;
%     %   \textbf{return } $\Sigma_{T}$;
% \end{algorithm}
\begin{algorithm}[ht]
    \caption{ Depth Descent Synchronization for $\SO(D)$}\label{alg:dds}
    \begin{algorithmic}
    \REQUIRE{$\bR(0)$, number of iterations $T$, selection rule $\cS$, $\eta \in (0,\eta^\star(D)]$ step size, $\beta$ depth parameter}
    \FOR{$t = 1, \ldots, T$}
    \STATE{$j = t \mod n$}
    \STATE{$\bv_j(t) =  \cS_{\bR}( D_{\beta} (\{\Log_{\bR_j(t)} \bR_{jk}\bR_k(t): k \in E^j\}))$}
    \STATE{$\bR_j(t+1) \gets \Exp_{\bR_{j}(t)} ( \eta \bv_j(t))$}
    \STATE{$\bR_k(t+1) \gets \bR_k(t), k \neq j$}
    \ENDFOR
    \RETURN{$\bR(T)$}
    
    \end{algorithmic}
    %   \textbf{return } $\Sigma_{T}$;
\end{algorithm}

{

One benefit of DDS is that there is no need to tune the step size, which is due to the local convexity properties of the manifold. As we outline in our main theorem, the step size can be directly selected from the guidance of~\cite{afsari2013convergence}. 

As for computational complexity, at least in low-dimensions, depth regions can be calculated efficiently for small datasets~\citep{liu2019fast}. In particular, the most straightforward algorithm involves an exhaustive search over hyperplanes spanned by $D$-subsets of $\cX$ that cut off $\beta n$ points in $\cX$. The time complexity of this method for $\mathbb{R}^3$ is $O(n^3 \log(n))$, and so could be used for moderately sized datasets. Translating this to our problem in $\SO(3)$, the time complexity for updating $j \in [n]$ is $O(n_j^3 \log(n_j))$, and so we see that this method is efficient for sparser graphs. In an Erd\"os-R\'enyi model, the complexity to update all $n$ rotations is $O(n^4 p^3)$, where $p$ is the Erd\"os-R\'enyi parameter.

As discussed in the introduction, the time complexity of the DDS algorithm is not necessarily more efficient than that of~\cite{lerman2019robust}. Indeed, the complexity of their message-passing algorithm is $O(n^3)$ for a single update to all rotations, while for our method it is $O(\sum_j n_j^3 \log(n_j))$ for $\SO(3)$ (and much larger for higher dimensions). Therefore, DDS has better complexity for sparse graphs, while~\cite{lerman2019robust} has better complexity for dense graphs. On the other hand, our complexity is uniformly better in $\SO(2)$, where depth contours can be easily found in $O(n \log(n))$ by sorting, and it thus takes $O(n^2 \log(n))$ time to update all nodes. Finally, we also note that our method has better scaling in terms of memory usage: the multiple rotation averaging scheme takes $O(n_j)$ memory while the message-passing scheme takes $O(n^3)$.

\subsubsection{The Approximate DDS Algorithm} 

To make the DDS algorithm more computationally efficient, we employ a few strategies to develop an approximate DDS algorithm. 

First, so that we do not need to resort to computing full depth contours, we instead take the average of the deepest points in the set $\{\Log_{\bR_j(t)} \bR_{jk}\bR_k(t): k \in E^j\}$ as our update direction at each iteration. 
%It is readily apparent that, as long as its depth is at least $\beta n_j$, this is a valid selection rule for $\cD_\beta(\{\Log_{\bR_j(t)} \bR_{jk}\bR_k(t): k \in E^j\})$. Indeed, notice that the set $\cD_\beta(\{\Log_{\bR_j(t)} \bR_{jk}\bR_k(t): k \in E^j\})$ is a convex polyhedron with vertices in $\{\Log_{\bR_j(t)} \bR_{jk}\bR_k(t): k \in E^j\}$, and so if the polyhedron contains nonzero values at least one vertex must be nonzero.

Second, to avoid computation of the full depth of every point in this set, we instead use an approximation of depth based on sampling. Suppose that we wish to approximate the depth of the vectors in $\cY = \{\by_1, \dots, \by_{n}\}\subset \R^D$ with respect to $\cY$. We can sample a set of vectors $\bu_1, \dots, \bu_m \in S^{D-1}$. Then, for each $\by_i \in \cY$, the approximation of depth, $\widetilde{\depth}$, is 
\begin{equation}\label{eq:approxdepth}
	\widetilde{\depth}(\by_i, \cY) = \min_{j\in [m]} \min\Big( \#(\{\by_k\in \cY : \by_k^\top\bu_j \geq \by_i^\top \bu_j \}),\#(\{\by_k \in \cY : \by_k^\top\bu_j \leq \by_i^\top \bu_j \})\Big).
\end{equation}
Notice that $\widetilde{\depth}$ replaces the minimum over $\bu \in S^{D-1}$ in~\eqref{eq:depth} by the minimum over the discrete set of vectors $\{\bu_1, \dots, \bu_m\}$.
For $\R^3$ (which corresponds to the tangent space for $\SO(3)$), computation of the full depth for all points would take $O(n_j^3)$ time, where for each $\bx_i$ one would need to search over all planes defined by triplets $\bx_i, \bx_j, \bx_k$, for distinct $i, j, k$. On the other hand,  the computation of the approximate depth using~\eqref{eq:approxdepth} takes $O(n_j^2 m)$ and can be more efficiently implemented due to the fact that the $\bu_1, \dots, \bu_m$ are shared between all $\by_i$. 

The approximate DDS algorithm is given in Algorithm~\ref{alg:adds}. Here, we use the notation $S_{\bR_{j}(t)} \SO(D)$ for the set of all unit vectors in $T_{\bR_j}\SO(D)$.

%Second, to avoid needing to input the depth parameter $\beta$, we instead just select the deepest nonzero point at each iteration and do not check if its depth is at least $\beta n_j$. In practice, we observe that this method still performs well despite not having this extra check.

\begin{algorithm}[ht]
    \caption{ Approximate Depth Descent Synchronization for $\SO(D)$}\label{alg:adds}
    \begin{algorithmic}
    \REQUIRE{$\bR(0)$, number of iterations $T$, selection rule $\cS$, $\eta \in (0,1)$ step size, $m$: number of depth vectors}
    \FOR{$t = 1, \ldots, T$}
    \STATE{$j = t \mod n$}
    \STATE{$\cY(t) = \{\Log_{\bR_j(t)} \bR_{jk}\bR_k(t): k \in E^j\}$}
    \STATE{$\bu_1, \dots, \bu_{m} \overset{i.i.d.}{\sim} \mathsf{Unif}( S_{\bR_{j}(t)} \SO(D))$}
    \STATE{$\bv_j(t) =  \argmax_{\by \in \cY(t)}\Big[ \min_{\bu_i}\Big( \min(\#\{\by' \in \cY(t) : \bu_i^\top (\by' -\by) \geq 0 \},\#\{\by' \in \cY(t):\bu_i^\top (\by' -\by) \leq 0 \})\Big) \Big]$}
    \STATE{$\bR_j(t+1) \gets \Exp_{\bR_{j}(t)} ( \eta \bv_j(t))$}
    \STATE{$\bR_k(t+1) \gets \bR_k(t), k \neq j$}
    \ENDFOR
    \RETURN{$\bR(T)$}
    
    \end{algorithmic}
    %   \textbf{return } $\Sigma_{T}$;
\end{algorithm}

}

\subsection{{Exact Recovery for DDS}}
\label{subsec:ddsrecov}

For many nonconvex methods, good initialization is quite important~\citep{li2019rapid,maunu2019well,qu2019convolutional,chi2019nonconvex,qu2020finding}. We only obtain a recovery result for DDS if we can initialize in a suitable neighborhood of $\bR^\star$.
\begin{definition}\label{def:neighborhood}
	A set of $n$ rotations $(\bR)$ lies within a $\rho$-neighborhood of $(\bR^\star)$ if the normalization products $\bR_j^{\star \top} \bR_j$ all lie in a ball of radius $\rho$. That is, there exists a $\bC \in \SO(D)$ such that 
	$$ d(\bR_j^{\star \top} \bR_j, \bC) < \rho, \ \forall  j = 1,\dots, n.$$
\end{definition}
With this terminology, our main assumption is that that we can initialize in a $\pi/2$-neighborhood of $(\bR^\star)$. This is the analog of Assumption~\ref{assump:initso2} for $\SO(D)$.
\begin{assumption}\label{assump:local}
	The initial set of rotations $(\bR(0))$ for our algorithm lies within a $\pi/2$-neighborhood of $(\bR^\star)$.
\end{assumption}
As we discuss in Section \ref{subsec:assumpdiscuss}, we believe that this assumption is not so restrictive, and we later give some intuition for how this might occur in a real scenario.

The following theorem constitutes the main theoretical result of this work. It states that, with proper initialization and well-connectedness, the recovery threshold of Algorithm~\ref{alg:dds} is $1/(D(D-1)+2)$. As examples when $\zeta = 1$, in the case of $\SO(2)$, Theorem~\ref{thm:sodrecovery} yields a corruption threshold of $\alpha_0 < 1/4$. In the case of $\SO(3)$, Theorem~\ref{thm:sodrecovery} yields a corruption level of $\alpha_0 < 1/8$.

% \begin{theorem}\label{thm:sodrecovery}
% 	Suppose that $\alpha_0 < 1/(D(D-1) + 2)$, Assumptions~\ref{assump:local} and~\ref{assump:weakwellconnect} hold, and $[(\bR(t))]_{t \in \nats}$ is generated by~\eqref{eq:depthseq} with $\beta = 1/(D(D-1)+2)$. Further assume in the case of $D=2,3$ that $\eta \in (0,1]$, and in the case of $D>3$ that $\eta$ is chosen according to Theorem 4.2 of~\cite{afsari2013convergence}.
% 	 Then, $d(\bR_j^{\star \top} \bR_j(t), \bR_k^{\star \top} \bR_k(t)) \to 0$ for all $j,k$, and the DDS algorithm exactly recovers $(\bR^\star)$.
% \end{theorem}

{
\begin{theorem}\label{thm:sodrecovery}
	Suppose that $\alpha_0 < \zeta/(D(D-1) + 2)$, Assumptions~\ref{assump:local} and~\ref{assump:weakwellconnect} hold, and $[(\bR(t))]_{t \in \nats}$ is generated by~\eqref{eq:depthseq} with $\beta = \zeta/(D(D-1)+2)$. Further assume in the case of $D=2,3$ that $\eta \in (0,1]$, and in the case of $D>3$ that $\eta$ is chosen according to Theorem 4.2 of~\cite{afsari2013convergence}.
	 Then, $d(\bR_j^{\star \top} \bR_j(t), \bR_k^{\star \top} \bR_k(t)) \to 0$ for all $j,k$, and the DDS algorithm exactly recovers $(\bR^\star)$.
\end{theorem}
}

\begin{proof}[Proof of Theorem~\ref{thm:sodrecovery}]

    To aid in the proof, we denote the smallest ball enclosing our normalization products as
\begin{equation}
B(t) := \argmin_{B(\bC, \rho)} \rho, \text{ s.t. } \bR_1^{\star \top} \bR_1(t), \dots, \bR_n^{\star \top} \bR_n(t) \in \overline{B(\bC, \rho)}.
\end{equation}
The center of $B(t)$ is $\bC(t)$ and its radius is $r(B(t))$. Our goal will be to show that $r(B(t)) \to 0,\ t \to \infty$. 

    %Lemma~\ref{lemma:bdin} implies, in particular, that if $\alpha_0 \leq 1/(D(D-1) + 2)$ and $B(t)$ has radius less than $\pi/2$. Then, $\bR_j^{\star \top}\bR_j(t+1) \in \overline{B(t)}$. Therefore, the set $\{\bR_j^{\star \top} \bR_j(t)\}$ lie in a strictly convex ball for all $t$ and also that the radius is $B(t)$ is nonincreasing. Thus, to apply a monotonic convergence theorem, such as Theorem 3.1 of~\cite{Meyer76}, it remains to show that the sequence $[r(B(t))]_{t \in \nats}$ is strictly monotonic.
    
    The proof of the theorem is broken into three parts. In the first part, we prove that the sequence $[(\bR(t))]_{t \in \nats}$ remains in a nested sequence of balls. In the second part, we show that, after sufficiently many iterations, the radius of the smallest enclosing ball must shrink. We finish in the third part by appealing to a general convergence theorem for monotonic algorithms.
    
    \textbf{Part I: $B(t+1) \subseteq B(t)$:}
    First, we show that at time $t$, no matter which index is updated, the normalization products remain in $\overline{B(t)}$. This is true at $t=0$, so assume that it is true at a time $t$.  Let $j = t \mod n$ and consider the pairwise measurements in the tangent space at $\bR_j(t)$: for each $k \in E_g^j$, the corresponding point in the tangent space is given by $\Log_{\bR_j(t)}(\bR_j^{\star} \bR_k^{\star \top} \bR_k(t))$. By assumption, we have that $\bR_j^{\star} \bR_k^{\star \top} \bR_k(t) \in \bR_j^\star \overline{B(t)}$ for all $k$. Since $\alpha_0 < \zeta/(D(D-1) + 2)$ and $\beta =\zeta/(D(D-1) + 2)$, 
    $$\cD_{\zeta/(2D+2)}(\{\Log_{\bR_j(t)}(\bR_{jk} \bR_k(t)) : k \in E^j\}) \subset \mathsf{conv}(\{\Log_{\bR_j(t)}(\bR_{jk} \bR_k(t)) : k \in E_g^j\}),$$
    since the set in the right-hand side of the display contains more than a $1-\beta$ fraction of points.
     We can now apply Theorem 3.7 of~\cite{afsari2013convergence}. This follows from the fact that the update direction  $\bv_j(t)$ is the gradient of the Frech\'et mean function for a weighted combination of $\{\Log_{\bR_j(t)}(\bR_{jk} \bR_k(t)) : k \in E_g^j\}$ (since it lies in the convex hull of of these points). More formally, letting $m = \#(E_g^j)$, since  $\bv_j(t) \in \mathsf{conv}(\{\Log_{\bR_j(t)}(\bR_{jk} \bR_k(t)) : k \in E_g^j\})$, there exist weights $a_1, \dots, a_m$ such that
    \begin{align*}
    \bv_j(t) &= \sum_{i=1}^m a_i \Log_{\bR_j(t)} \bR_{j{k_i}} \bR_{k_i}(t) \\
    &= \sum_{i=1}^m a_i \Log_{\bR_j(t)} \bR_j^\star \bR_{k_i}^{\star \top} \bR_{k_i}(t) \\
    &= -\mathsf{grad} \sum_{i=1}^m a_i d^2(\bR_{k_i}^{\star \top} \bR_{k_i}(t),\cdot) \Big|_{R_{j}(t)}.
    \end{align*}
    Therefore, choosing the step size as in \cite{afsari2013convergence} implies that $\bR_j^{\star \top} \bR_{j}(t+1) \in \overline{B(t)}$, and further that $\bR_j^{\star \top} \bR_{j}(t+1) \in {B(t)}$ when  $\bv_j(t) \neq 0$. In turn, this implies that $\overline{B(t+1)} \subseteq \overline{B(t)}$ and, if $\bR_{j}(t) \in B(t)$, then $\bR_{j}(t+1) \in B(t)$ as well (i.e., interior points cannot move to the boundary).
    
    In the case of $\SO(3)$, Theorem 3.7 of~\cite{afsari2013convergence} tells us that choosing $\eta \in (0,1]$ suffices. The case of $\SO(D)$ for $D>3$ is dealt with in a similar way using Theorem 4.2 of~\cite{afsari2013convergence}.
    
	% Therefore, as long as $r(B(t)) \leq B(0)$, the set $B(t)$ remains totally convex. 
	
	\textbf{Part II: $r(B(s+\Delta_s)) < r(B(s))$:}
	We now must show that after sufficiently many iterations, the radius of $B(t)$ strictly decreases. To this end, fix a time $s$. At this time, at least one normalization product $\bR_j^{\star \top} \bR_j(s)$ must lie on the boundary $\partial B(s)$. 
	%We will show that the radius of $\overline{B(t)}$, $r(B(t))$, is strictly decreasing in time. The notation $\overline{\cdot}$ emphasizes that this is a closed ball. 
	For convenience, define the index set $J(s)$ of boundary rotations at time $s$ by
	$$J(s):= \bset{j:\bR_j^{\star\top} \bR_j(s) \in \partial B(s)}.$$
	   We will show that there exists a $\Delta_s > 0$ such that at some future time $s + \Delta_s$,
    \begin{equation}\label{eq:radmonotonic}
    	r(B(s+\Delta_s)) < r(B(s)).
    \end{equation}
    To this end, pick a direction $w$ uniformly at random from $T_{\bC(s)} \SO(D)$ such that $\|w\|_{\bC(s)} = 1$. This vector separates $T_{\bC(s)}$ into two halfspaces, and thus partitions $B(s)$ into two halves. One of these halves contains at most $n/2$ points $\Log_{\bC(s)} (\bR_j^{\star \top} \bR_j(s))$, and we will denote the corresponding halfspace of $T_{\bC(s)}\SO(D)$ by $\cH$. Since the direction $w$ is chosen uniformly at random, there are no points on the boundary of this halfspace with probability 1.%: $\partial \cH\cap \{\Log_{\bC(s)} (\bR_k^{\star \top} \bR_k(t)):k \in E^j\} = \emptyset$.% We will show, after a sufficiently long time $\Delta_s$, all of the rotations in the halfspace with less move into the interior.

    %The argument follows the logic of the proof of Theorem~\ref{thm:tasrecovery_general}.
    Let $K(s)$ denote the set 
    $$K(s):= \{k : \bR_k^{\star \top} \bR_k(s) \in \Exp_{\bC(s)} (\cH) \cap \partial B(s) \},$$
    that is, the set of indices corresponding to boundary normalization products at time $s$. 
    At each time $t = s + m$ for $m>0$, if none of the normalization products $\bR_k^{\star \top} \bR_k(t)$, $k \in K(s)$, lie in $\partial B(s)$, then set $\Delta_s = m$ and we can apply Lemma~\ref{lemma:shrinkball} to yield~\eqref{eq:radmonotonic}.
    
    Otherwise, by Assumption~\ref{assump:weakwellconnect}, there is at least one index $k \in K(s)$ such that $\bR_k^{\star\top} \bR_k(s)$ is in $\partial B(s)$  and
    $$ \#\big(E^k \setminus K(s) \big) > \#\big(E^k \cap  K(s) \big).$$
    Suppose that we update this index $k$ at time $t = s+m$.
    We are in a situation where we can apply Lemma \ref{lemma:depthcont}, with 
    $$\cX = \{\Log_{\bR_j(t)} \bR_{jk} \bR_k(t) : k \in E^j\}, \quad \cY = \{\Log_{\bR_j(t)} \bR_{jk} \bR_k(t) : k \in E^j_g\},$$
    which yields that  $\bv_j(s+m) \in \mathrm{conv}(\cY)$ and  $\bv_j(s+m) \neq 0$.
    Thus, for sufficiently small $\eta(s+m)$, $\bR_j(s+m+1) \in B(s)$. Repeating this sequentially for all elements of $K(s)$, there must exist a $\Delta_s$ such that 
    $$ \bR_j^{\star \top} \bR_j(s+\Delta_s) \in B(s), \ \forall j \in K(s).$$

    We are left in a situation where $\bR_1^{\star \top} \bR_{1}(s + \Delta_s), \dots, \bR_n^{\star \top} \bR_{n}(s + \Delta_s) \in \overline{B(s)}$, and there exists a $w$ such that $\cH(\bC(s), -w)$ contains no boundary points. By appealing to Lemma~\ref{lemma:shrinkball}, we know that $(\bR(s + \Delta_s)$ lies in a ball of radius smaller than $r(B(s))$, and therefore $r(B(s))$ has a strictly monotonic subsequence.
    
    %For completeness, we recall the following notions related to general algorithms \citep{Meyer76}. A \emph{point-to-set map} $\cP$ over a space $S$ is a map $\cP:S \to 2^S$. An \emph{iterative algorithm} is determined by a point-to-set map $\cP:S \to 2^S$ and an initial value $x(0) \in S$. It generates a set-valued sequence such that $x(t+1) \in \cP(x(t))$ for all $t \in \mathbb{N}$. The algorithm is monotonic with respect to an auxiliary continuous nonnegative function $\phi: S \to \R_+$ if $\phi(y) \leq \phi(x)$ for any $y \in \cP(x)$. The algorithm is strictly monotonic if this inequality is strict whenever $x$ is not a fixed point of $\cP$.
    
    \textbf{Part III: Strict monotonicity implies convergence:}
    By \cite{mizera2002depth}, as long as $\beta \leq 1/(D(D-1)/2+1)$, the point to set mapping $$\bR_j(t) \mapsto D_{\beta} (\{\log_{\bR_j(t)} (\bR_{jk}\bR_k(t)) : k \in E^j\})$$ is non-empty and outer semicontinuous with respect to the empirical measure on $$\{\Log_{\bR_j(t)} \bR_{jk}\bR_k(t): k \in E^j\}.$$ Therefore, the associated algorithm~\eqref{eq:depthseq} is upper semicontinuous in the sense of Theorem 3.1 of~\cite{Meyer76}, and we obtain convergence of $\bR(t)$ to a fixed point of~\eqref{eq:depthseq}.
    
    We finish by examining fixed points of the algorithm~\eqref{eq:depthseq}. Suppose that $\bR$ is a fixed point of this sequence, such that  $\bv_j$ as defined by~\eqref{eq:vj} is zero. %Examining the application of Lemma~\ref{lemma:depthcont} above, we see that this only happens if, for all $j$, $d(\bR_j, \bR_j^{\star} \bR_k^{\star \top} \bR_k) = 0$ for at least $n_j/2$ measurements $k$. Together with the well-connectedness condition of Assumption~\ref{assump:weakwellconnect}, this implies that $d(\bR_j, \bR_j^{\star} \bR_k^{\star \top} \bR_j) = 0$ for all $j,k$.
    {
    The fixed point is characterized by
$d(\bR_j, \bR_j^{\star} \bR_k^{\star \top} \bR_k) = 0$ for at least $(1-\zeta/2)n_j$ measurements $k$. By $\zeta$-well-connectedness, for all subsets $J$ of size at most $n/2$, there is an index $j$ such that $\#\Big[ E^j \cap J\big) < (1-\zeta/2) n_j$, because otherwise Lemma~\ref{lemma:depthcont} would yield a nonzero update direction. Thus, there is an index $k \in E^j \setminus [J]$ such that $d(\bR_j, \bR_j^{\star} \bR_k^{\star \top} \bR_k) = 0$. This implies that $d(\bR_j, \bR_j^{\star} \bR_k^{\star \top} \bR_k) = 0$ for all $j,k$.}
    \qed
\end{proof}

\subsection{{ Discussion of Assumptions}}
\label{subsec:assumpdiscuss}

{
The $\zeta$-well-connectedness condition in Assumption~\ref{assump:weakwellconnect} bears some similarity to the notions of conductance and graph expansion. From the perspective of graph theoretical results, we note that a sufficient condition for Assumption~\ref{assump:weakwellconnect} with $\zeta = 1$ is for the conductance of the graph to be greater than or equal to 1/2. This follows from a simple pigeonhole argument. It is unclear if this condition holds for Erd\"os-R\'enyi graphs. Indeed, if one uses Cheeger's inequality, one would need the spectral gap to be greater than or equal to 1, but for Erd\"os-R\'enyi graphs one only expects this gap to concentrate around 1 in practice~\citep{hoffman2021spectral}.

A sufficient condition for~\eqref{eq:weakwellconnect2} is for the conductance to be bounded below by ${\zeta}/{2}$, which can be achieved with high probability for any fixed $\zeta<1$ by Erd\"os-R\'enyi graphs when $p \gtrsim \log(n)/n$ (see, for example,~\citet{hoffman2021spectral}), as well as expander graphs  (see, for example,~\citet{friedman2003relative}). 

Our result in Theorem~\ref{thm:sodrecovery} holds with high probability for the uniform corruption model discussed in Section~\ref{subsec:robnotions} with $q < \alpha_0 = \zeta / (D(D-1) + 2)$ and $p \gtrsim \log(n)/n$. This means that DDS achieves the information theoretically optimal rate with respect to $p$ in this model. Indeed, we can read~\eqref{eq:inforate} as $p = \Omega\Big(\frac{1}{(1-q)^2} \frac{\log(n)}{n}\Big)$. 
}

Assumption~\ref{assump:local} requires that we initialize the DDS algorithm so that the normalization products lie in sufficiently small ball. This can be achieved in practice for cameras whose orientations lie close enough together. That is, suppose that all of the rotations $\bR_1^{\star \top}, \dots, \bR_n^{\star\top}$ lie in $B(\bS, \rho)$ for all $j$, for some $\bS$ and $\rho < \pi/2$. Then, if the initial point for the DDS algorithm is chosen to be $(\bI, \dots, \bI)$, then it is not hard to see that 
\begin{equation}
    \bR_j^{\star \top} \bI = \bR_j^{\star \top} \in B(\bS, \rho), \quad \forall j \in [n].
\end{equation}
which directly shows that Assumption~\ref{assump:local} holds.
{Notice that $B(\bS, \pi/2)$ is a large ball that essentially makes up half of the manifold $\SO(3)$, since the distance from any point to its cut-locus is $\pi$. For example, if one considers reconstruction of an object from many images taken from points on a sphere that surrounds this object, then our requirement would essentially boil down to needing all of the images being taken from a single hemisphere.

We conjecture that one can weaken this initialization condition, to only require that neighboring normalization products are close to each other, but we leave weakening of this assumption to future work.
}
%Future work will precisely spell out when this initialization works versus when a spectral initialization works.

{
\section{Empirical Evaluation}
\label{sec:exp}

The algorithms we compare with are MRA and L1-MRA~\citep{hartley2013rotation}, IRLS after L1-MRA initialization~\citep{chatterjee2013efficient}, LTS~\citep{huang2019learning}, CEMP~\citep{lerman2019robust}, and MPLS~\citep{shi2020message}. Default parameters of all methods are used. For CEMP, the method computes the corruptions levels first to determine which edges are most corrupted. Then, using these corruption levels, it finds a minimum spanning tree, from which one can fix $\bR_1 = \bI$ and then propagate from this to find the other rotations along this tree. For LTS, we implement the truncation step with parameter $\gamma = 0.96$ and run for 40 iterations. The approximate DDS algorithm is run for 40 epochs (or passes over the data, which means we take $T=40n$) with a step size of $\eta = 0.7$ and number of depth vectors $m=20$.

We compute the distance between the estimated rotations $(\hat \bR)$ and $(\bR^\star)$ by first aligning them by solving
\begin{equation}
    \bS = \argmin_{\bS' \in \SO(3)} \sum_{i=1}^n \| \bR_i^\star - \hat \bR_i \bS\|^2.
\end{equation}
The error is then computed as
\begin{equation}
    \mathsf{err}(\hat \bR, \bR^\star) = \max_{i=1, \dots, n} d_{\angle}(\hat \bR_i \bS, \bR_i^\star).
\end{equation}
%We do not display runtimes since our code is not optimized. 
All algorithms take less than a minute to run on each individual dataset on a Macbook Air with a 1.6 GHz Dual-Core Intel Core i5 and 8 GB of RAM.

Figure~\ref{fig:exp_unifunif_errs} presents a first comparison of these algorithms on synthetic data. The model is the uniform corruption model, which is discussed in Section~\ref{subsec:robnotions}. The graph is Erd\"os-R\'enyi on $n=100$ nodes with varying parameter $p$, and each edge on this graph is corrupted with probability $q$. The underlying rotations, $(\bR_1^\star, \dots, \bR_n^\star)$, are distributed uniformly on $\SO(D)$. The bad measurements, $\bR_{jk}^b$, are also uniformly distributed on $\SO(D)$. For each set of parameters ($p=0.1, 0.2, \ldots, 0.5$ and $q=0.05, 0.1, \ldots, 0.3$), 10 datasets are generated and the color represents the mean of the $\log_{10}$-errors over these experiments. As we can see, the approximate DDS algorithm performs on par with the other most competitive methods (CEMP and MPLS). 
\begin{figure}[ht]
    \centering
    \includegraphics[width = .24\textwidth]{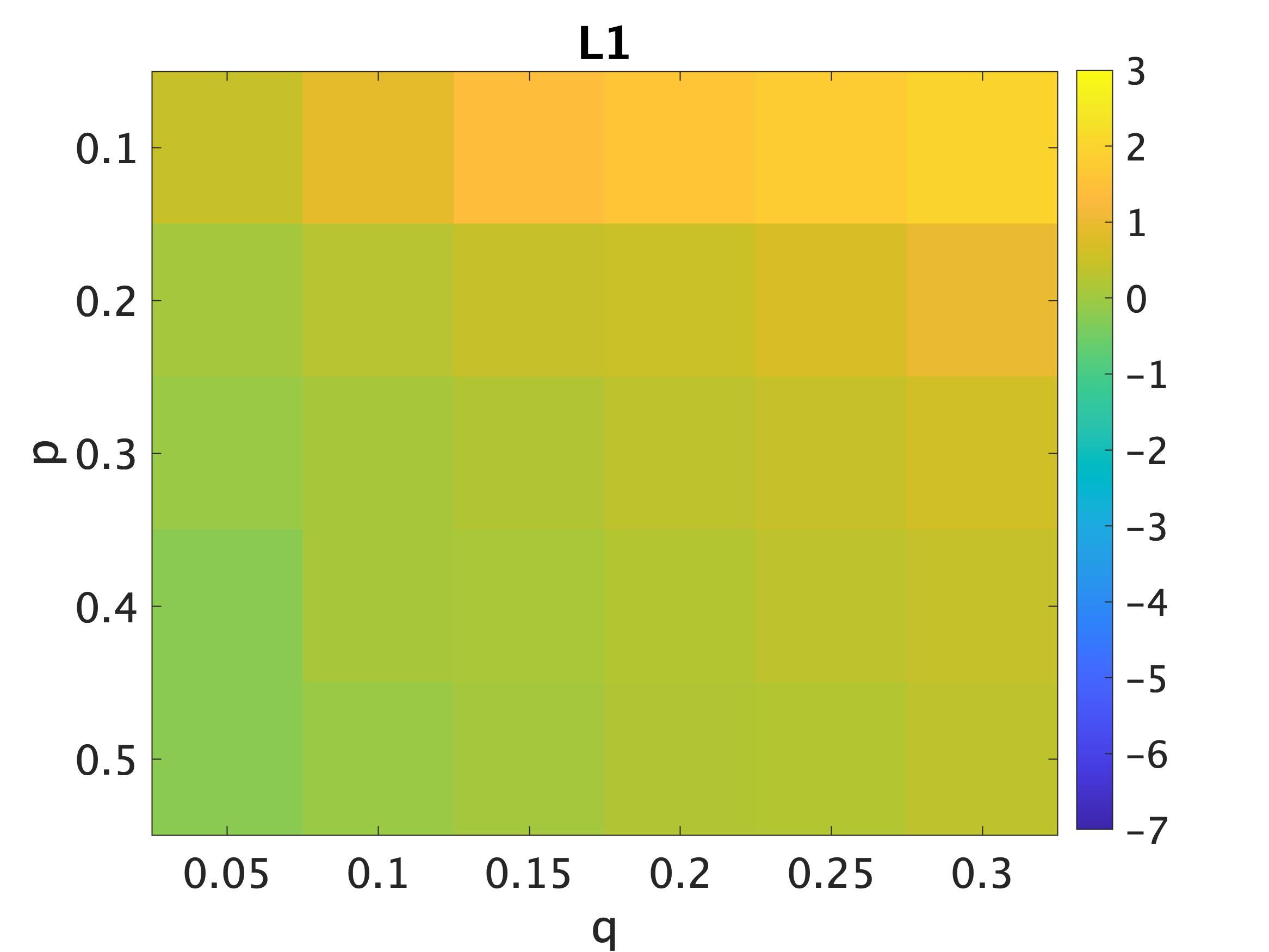}
    \includegraphics[width = .24\textwidth]{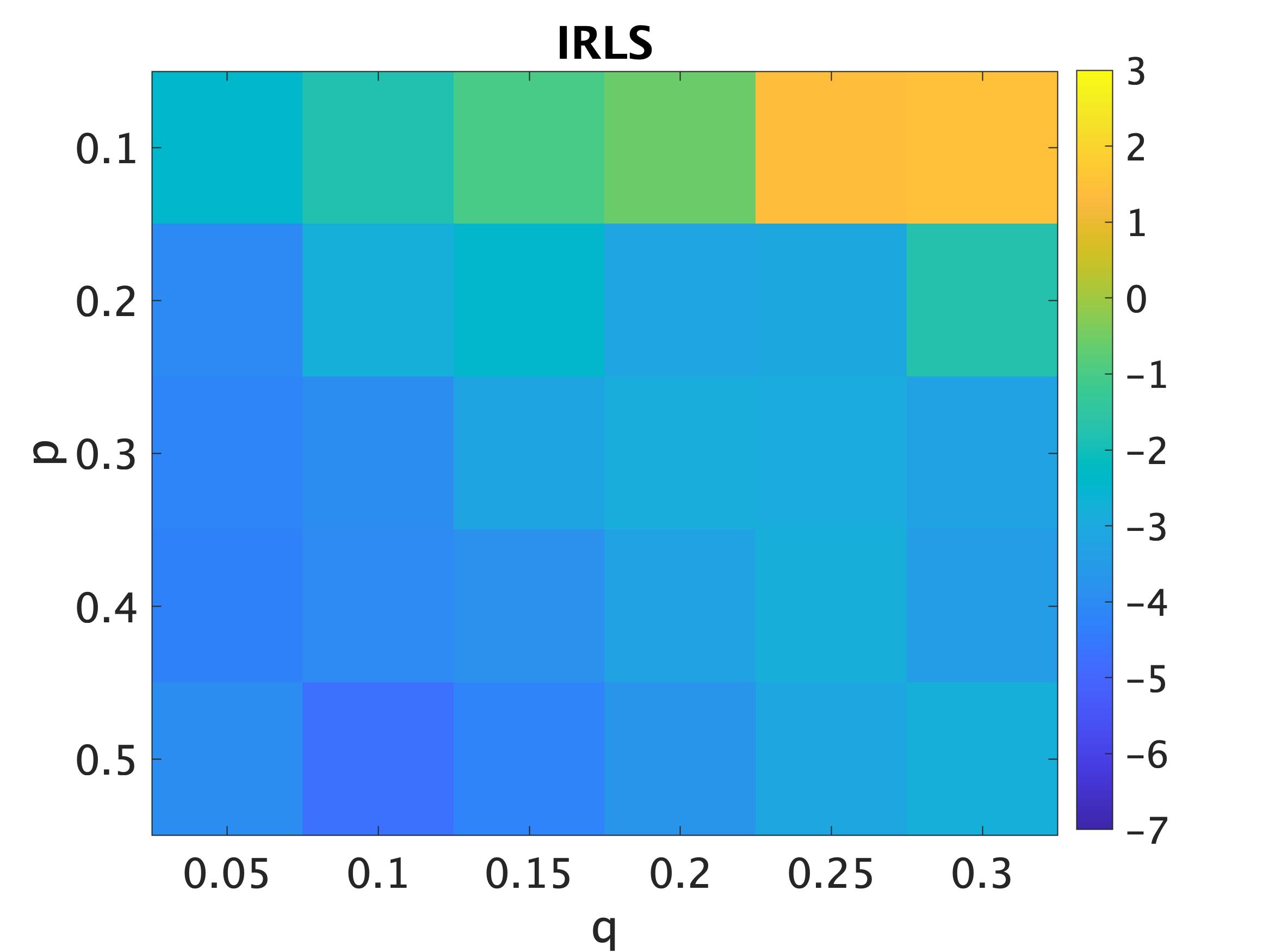}
    \includegraphics[width = .24\textwidth]{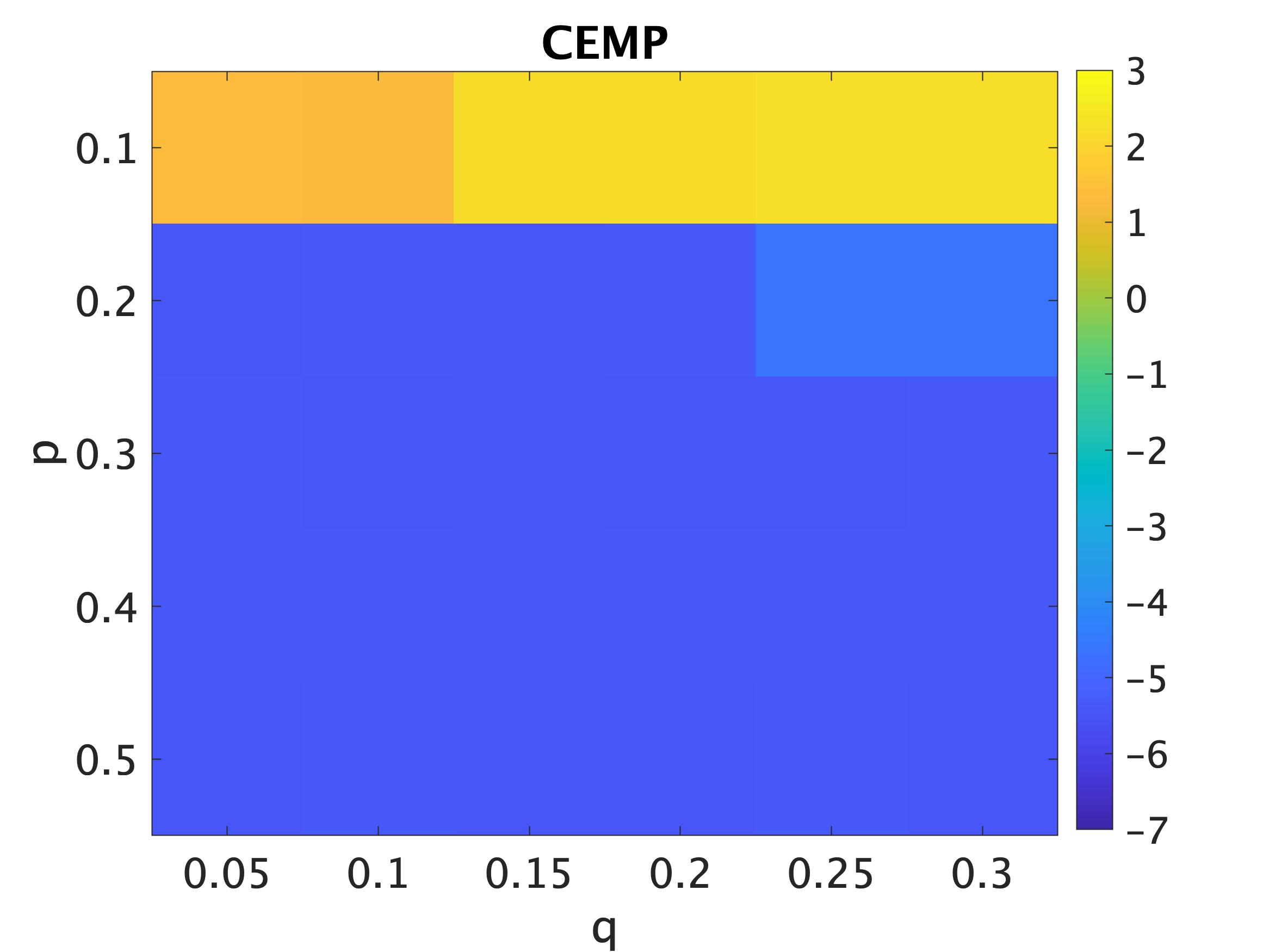}
    \includegraphics[width = .24\textwidth]{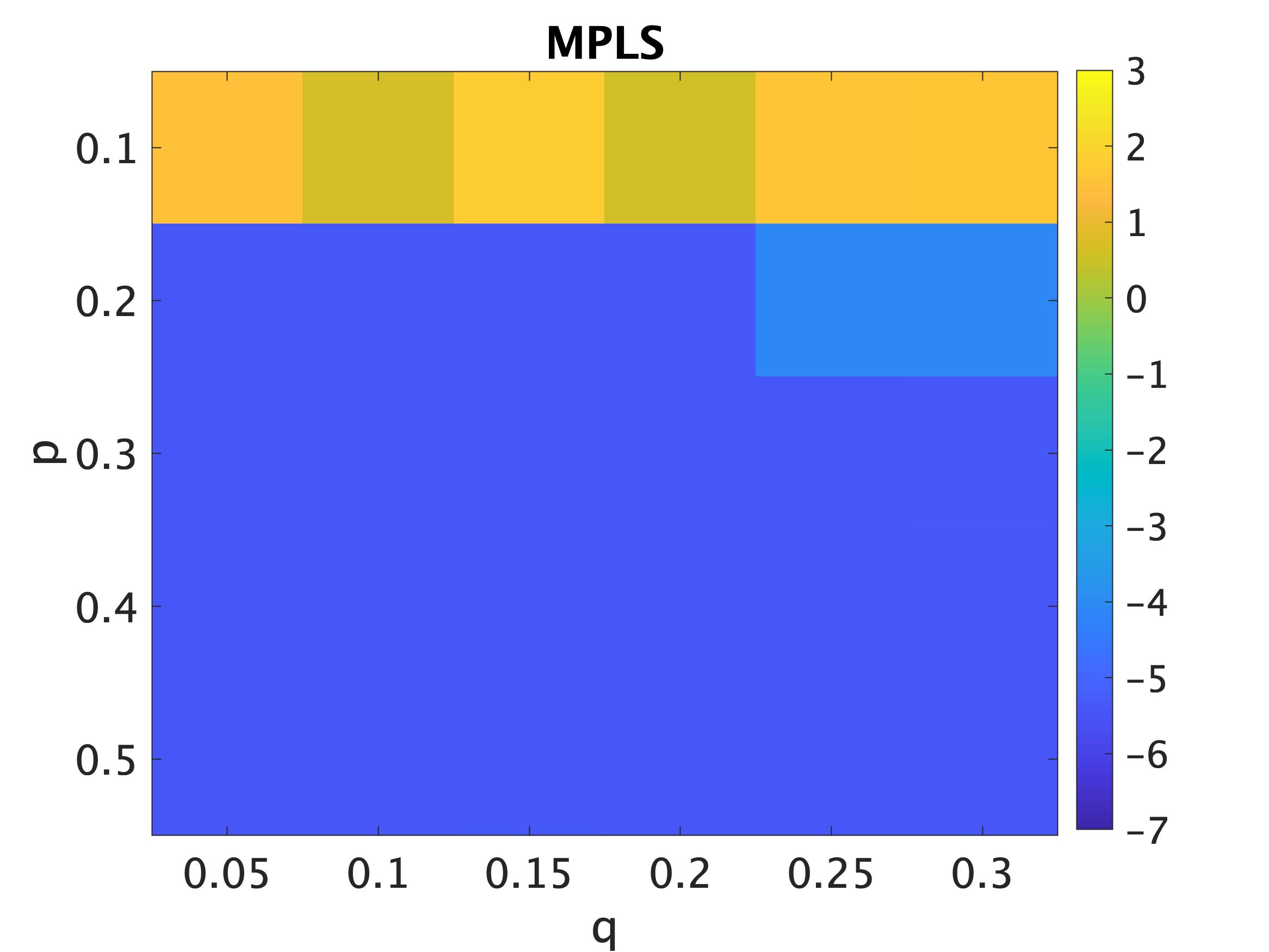}
    \includegraphics[width = .24\textwidth]{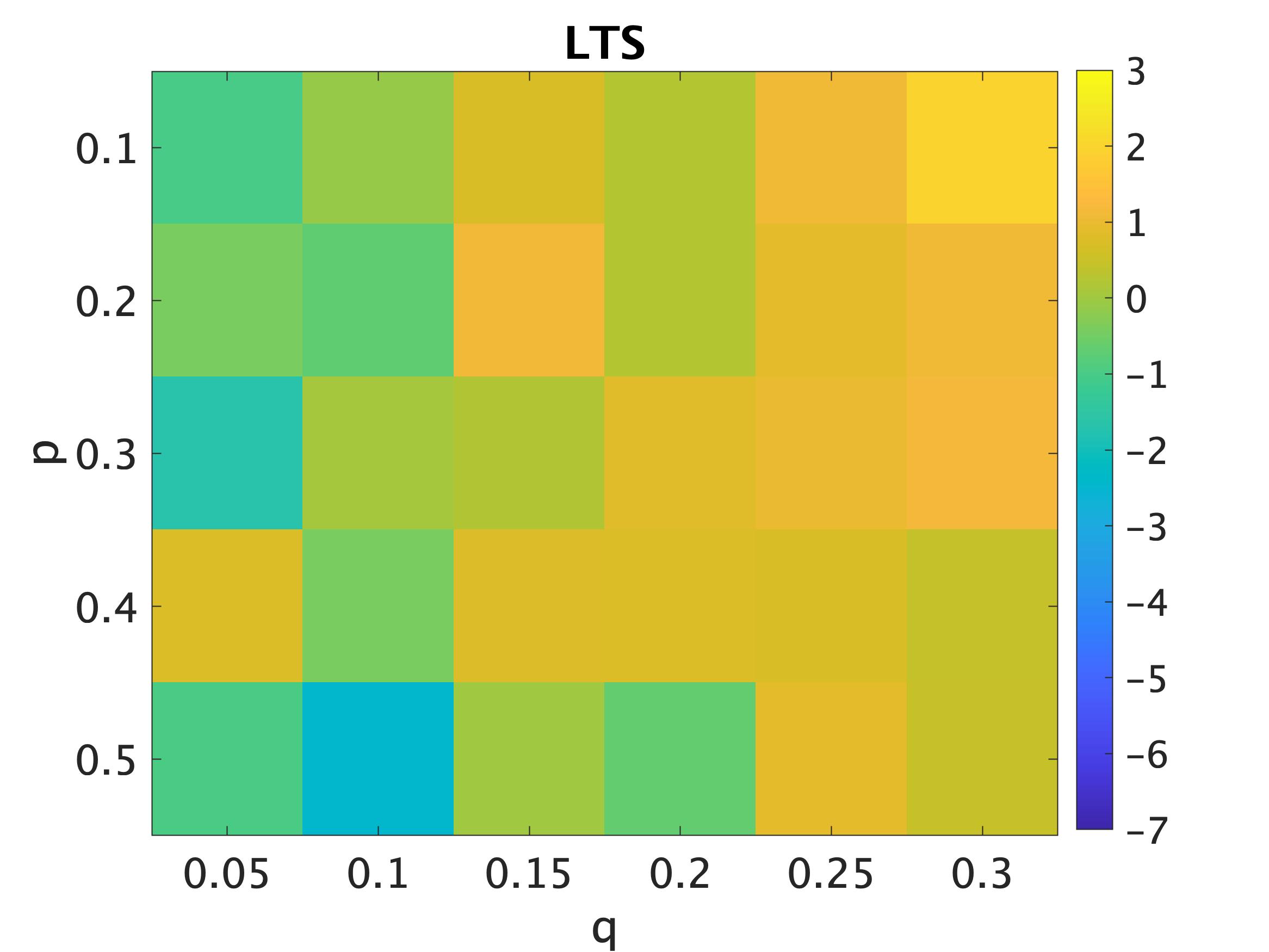}
    \includegraphics[width = .24\textwidth]{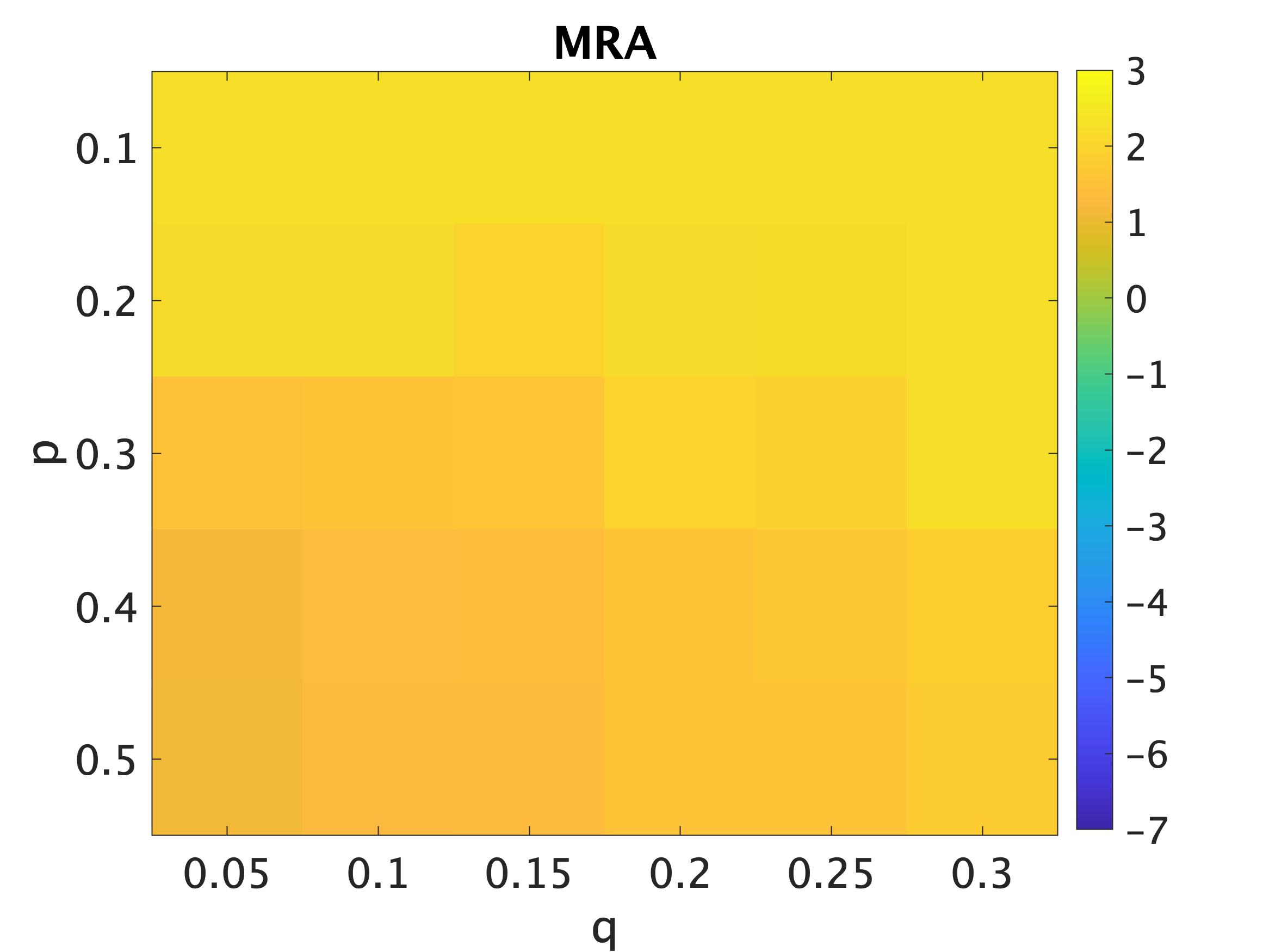}
    \includegraphics[width = .24\textwidth]{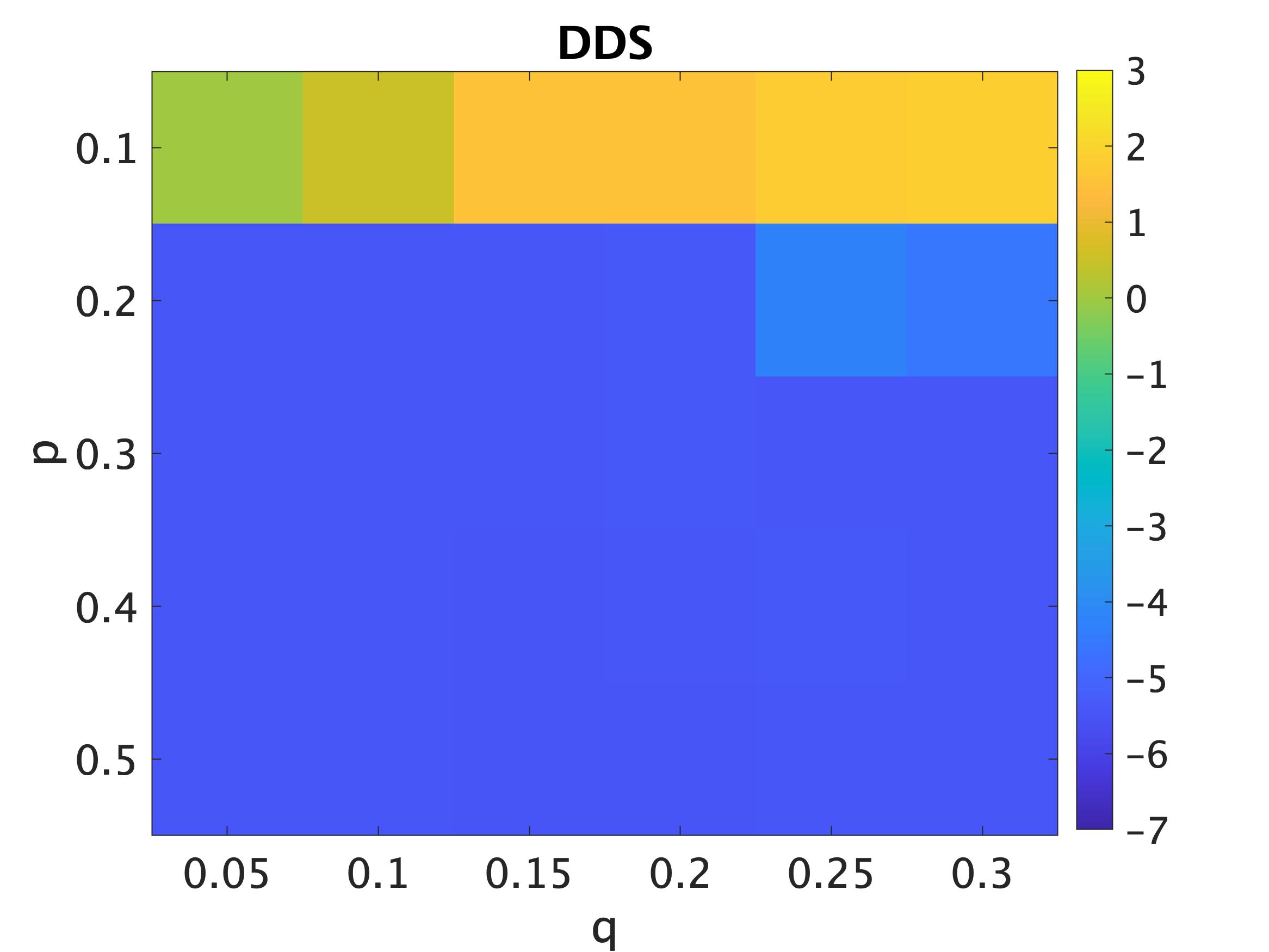}
    \caption{Rotation synchronization experiment with uniform outliers. Here, $p$ is the parameter of the Erd\"os-R\'enyi graph, and $q$ is the percentage of corrupted edges. The underlying rotations are distributed uniformly in $\SO(D)$, and the corrupted measurements of group ratios are also uniform on $\SO(D)$. The color represents the mean of the $\log_{10}$-errors over the 10 generated datasets.}
    \label{fig:exp_unifunif_errs}
\end{figure}

% \begin{figure}[ht]
%     \centering
%     \includegraphics[width = .5\textwidth]{plots/times_unifunif.jpg}
%     \caption{Rotation synchronization experiment with adversarial outliers. Here, $p$ is the parameter of the Erd\"os-R\'enyi graph, and $q$ is the percentage of corrupted edges which are uniformly distributed across this graph.}
%     \label{fig:exp_unifunif_times}
% \end{figure}

As a second experiment, Figure~\ref{fig:exp_lineline_errs} presents a more challenging adversarial example on synthetic data. Here, the outliers form a consistent set of measurements themselves, and a similar corruption model is discussed in Section 7.3 of~\cite{lerman2019robust}, although here we extend this to $\SO(3)$ and use a different model for the underlying rotations.

The graph is an Erd\"os-R\'enyi graph on $n=50$ nodes with parameter $p$, and each edge on this graph is corrupted with probability $q$. The ground truth rotations approximately come from a geodesic on $\SO(3)$, and the outliers are self-consistent measurements that come from (approximately) another geodesic on $\SO(3)$. The ground truth rotations are
\begin{equation}\label{eq:linein}
    \bR_{i}^\star = \Exp_{\bI}\Big( -s_i (\bv + \bxi_i)\Big),
\end{equation}
where $\bv$ is a fixed vector drawn uniformly from the sphere, $\bxi_i \sim N(\bzero,10^{-4} \bI)$, and $s_i = -1 + {2(i-1)}/{50}$. The outliers generated by pairwise measurements between another set of rotations
\begin{equation}\label{eq:lineout}
    \bR_{i}^b = \Exp_{\bI}\Big( -s_i (\bv' + \bxi_i')\Big),
\end{equation}
where again $\bv'$ is a fixed vector drawn uniformly from the sphere, $\bxi_i \sim N(\bzero,0.5 \bI)$, and $s_i = -1 + {2(i-1)}/{50}$. As before, for each set of parameters $p$ and $q$, 10 datasets are generated and the color represents the mean of the $\log_{10}$-errors over these experiments. As we can see again, the approximate DDS algorithm performs well in this experiment, and in fact it performs on par with the most competitive rotation synchronization algorithms (CEMP and MPLS).
\begin{figure}[ht]
    \centering
    \includegraphics[width = .24\textwidth]{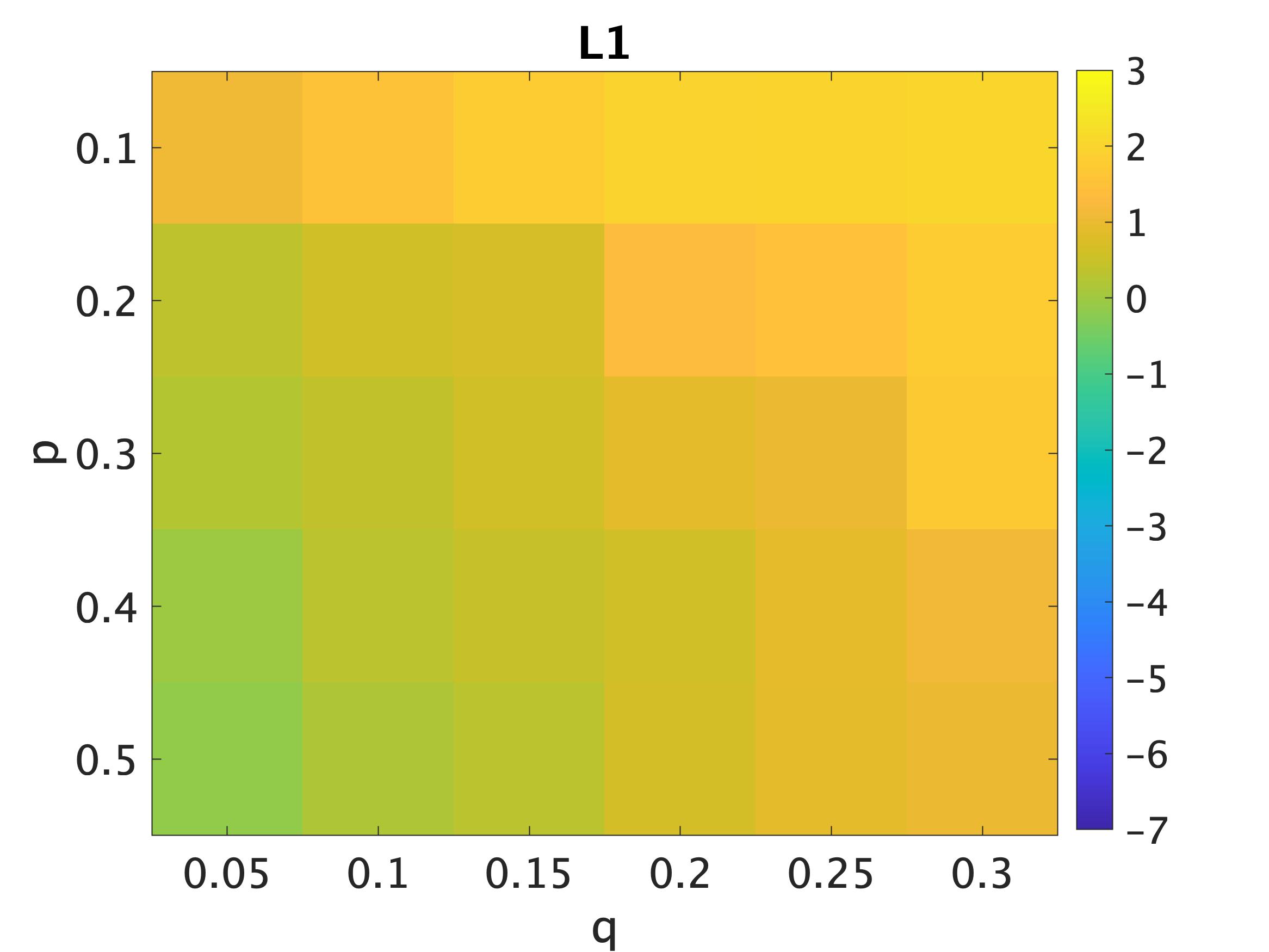}
    \includegraphics[width = .24\textwidth]{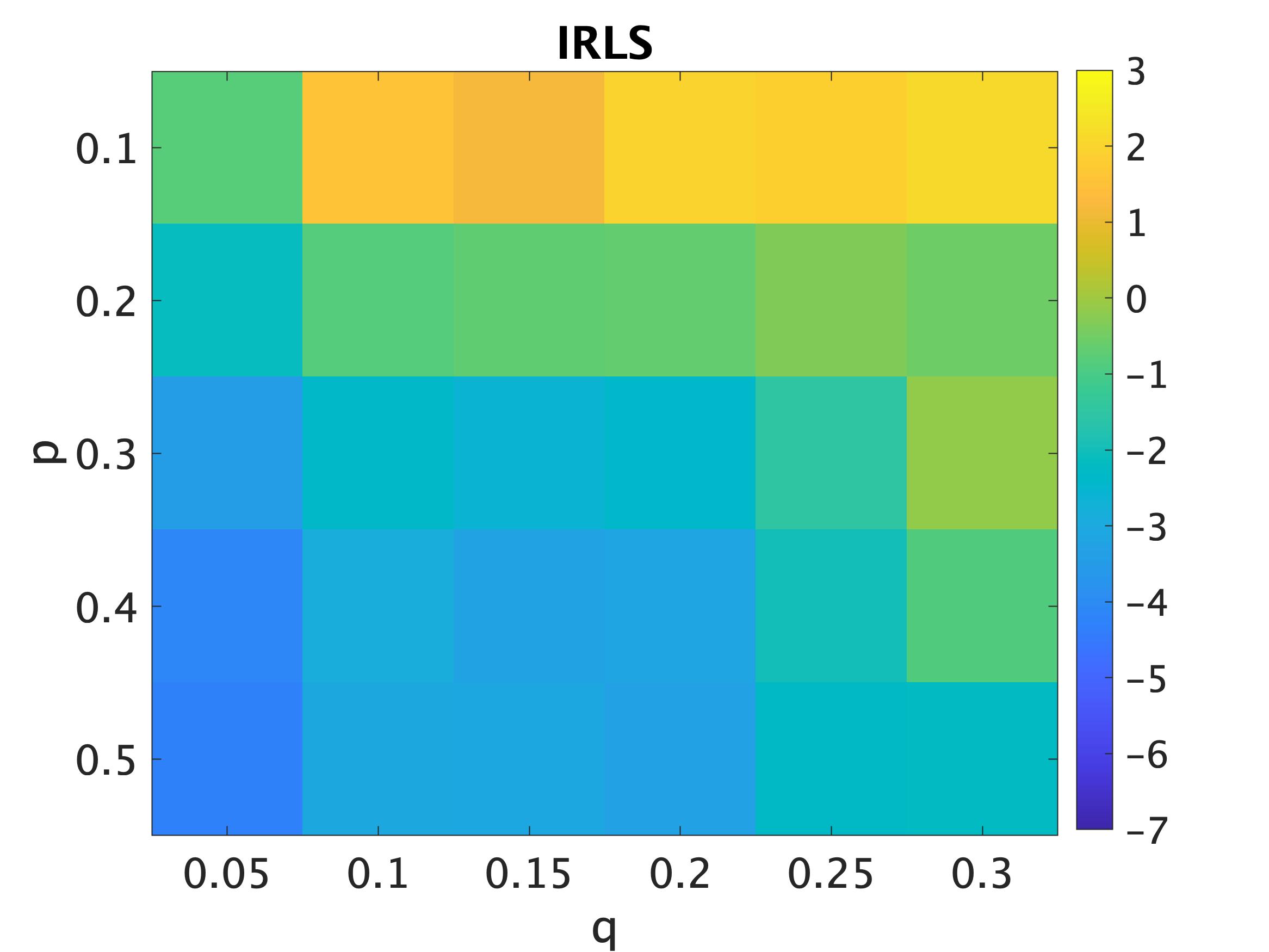}
    \includegraphics[width = .24\textwidth]{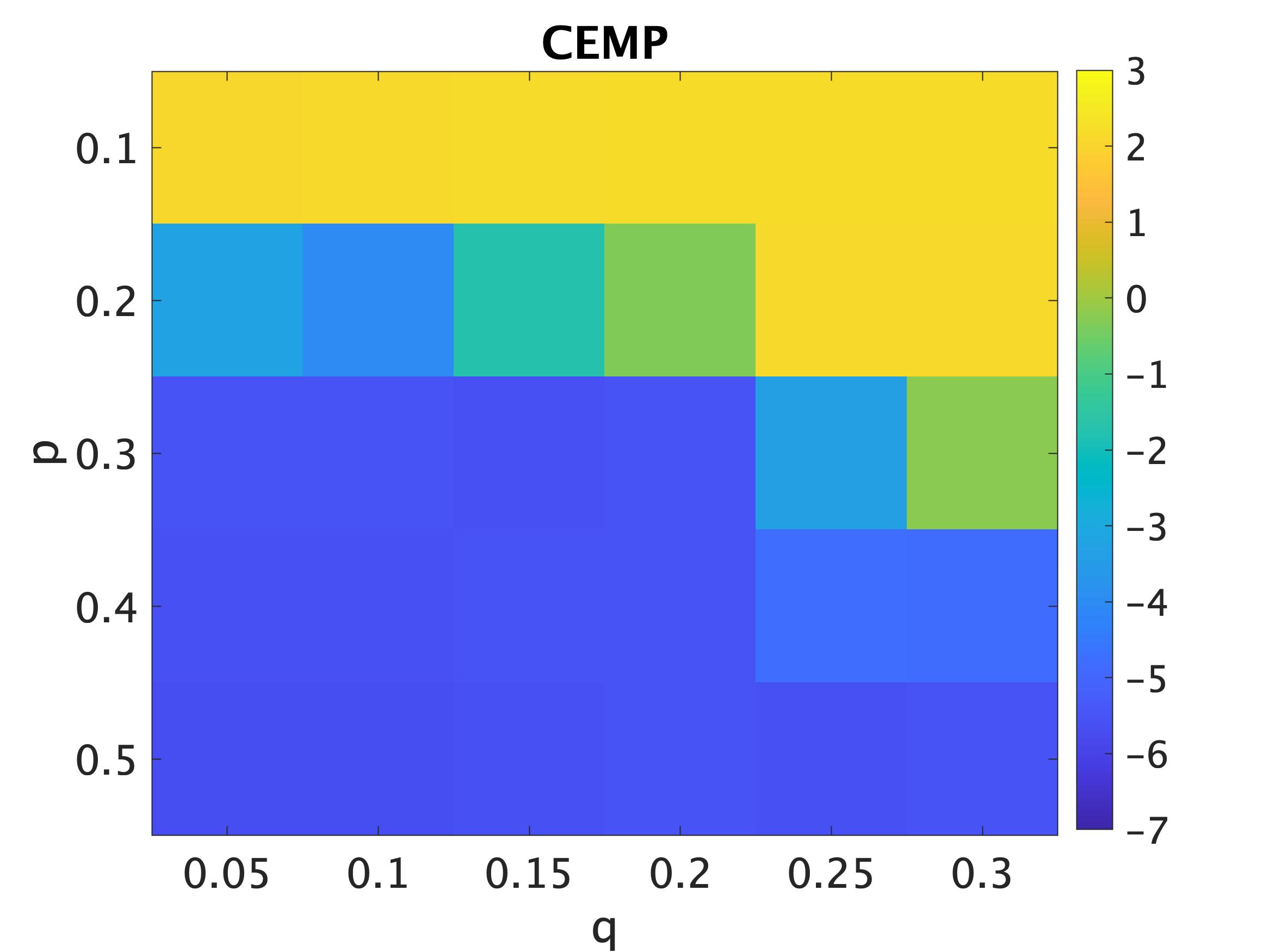}
    \includegraphics[width = .24\textwidth]{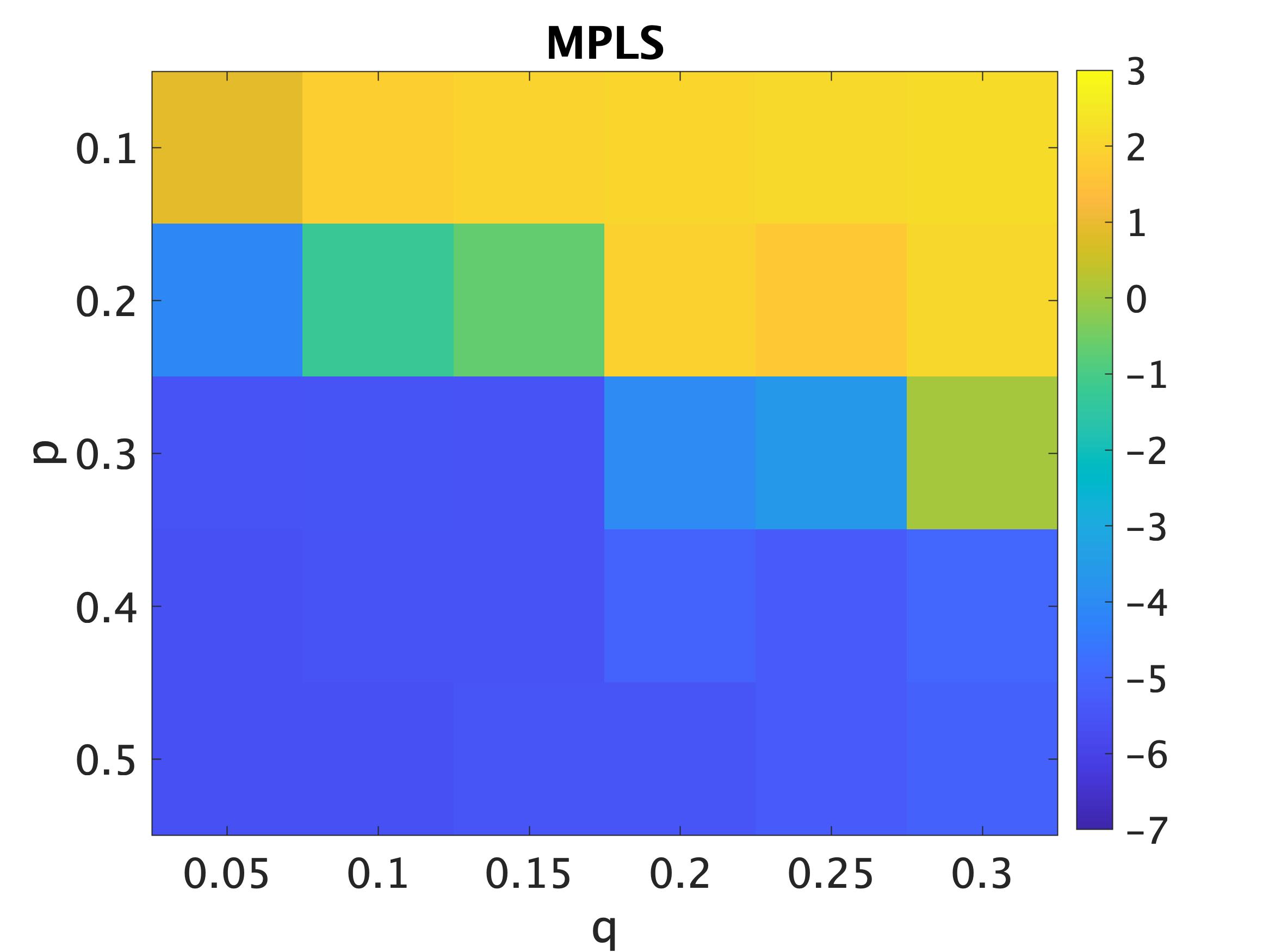}
    \includegraphics[width = .24\textwidth]{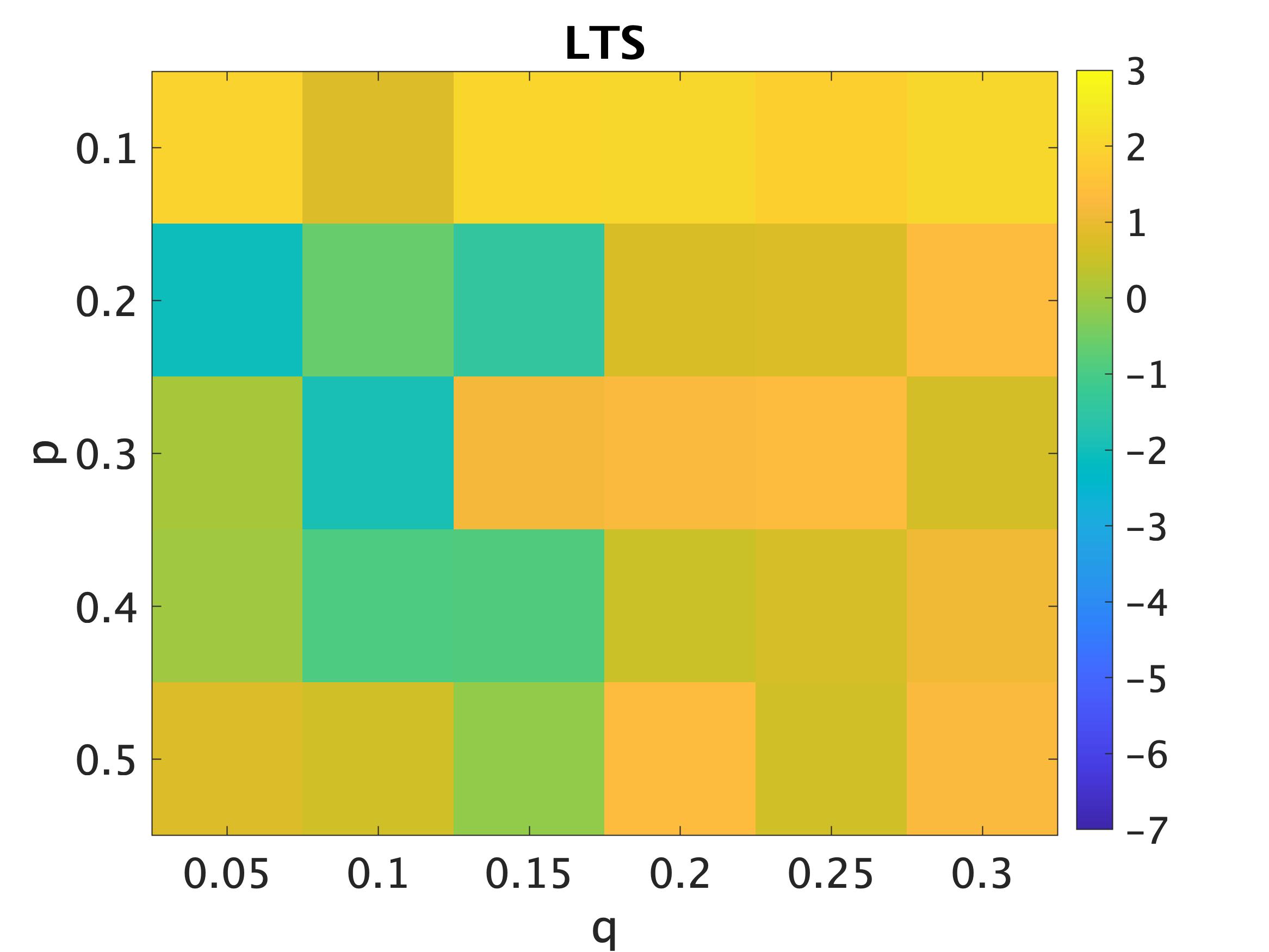}
    \includegraphics[width = .24\textwidth]{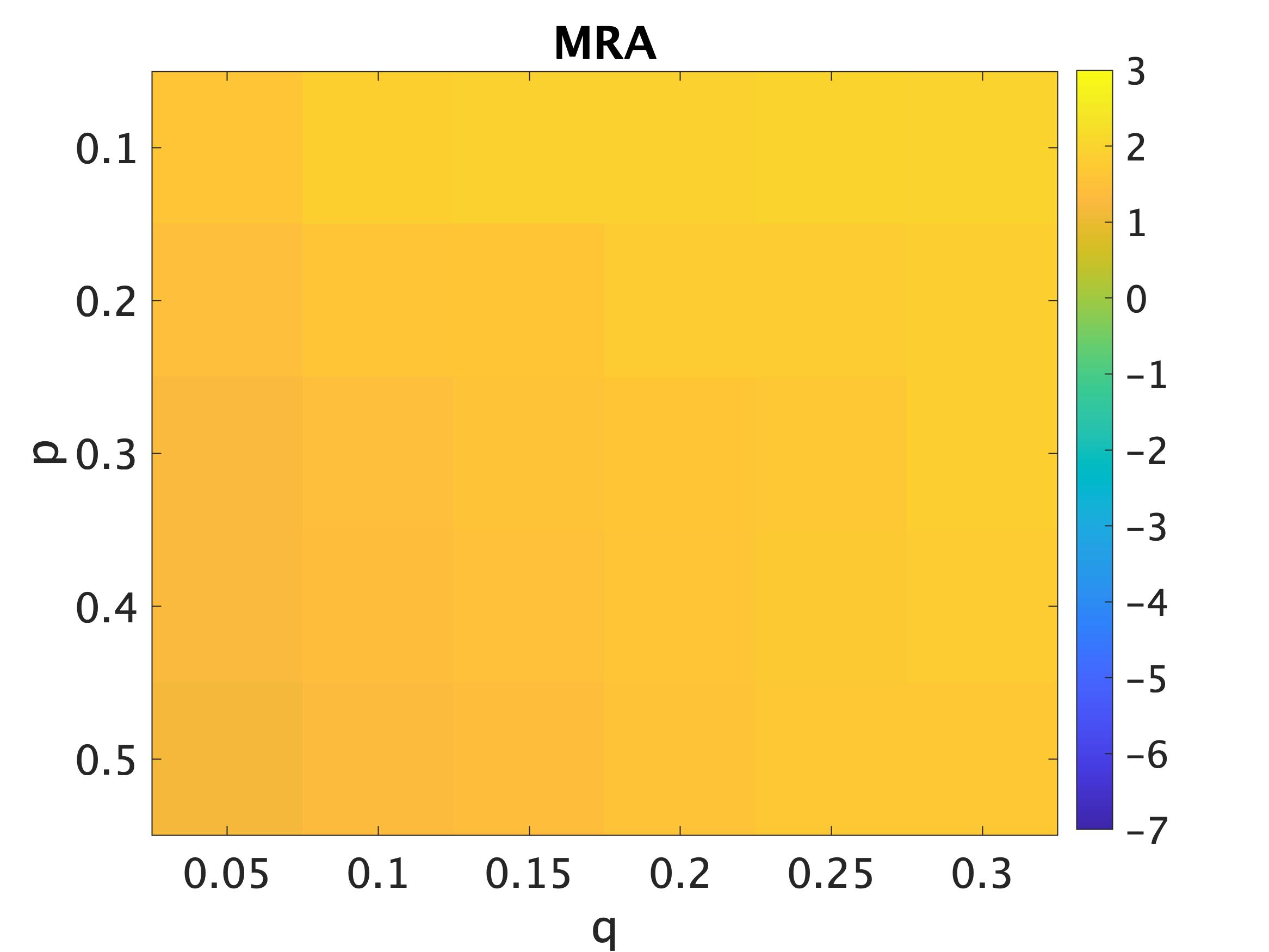}
    \includegraphics[width = .24\textwidth]{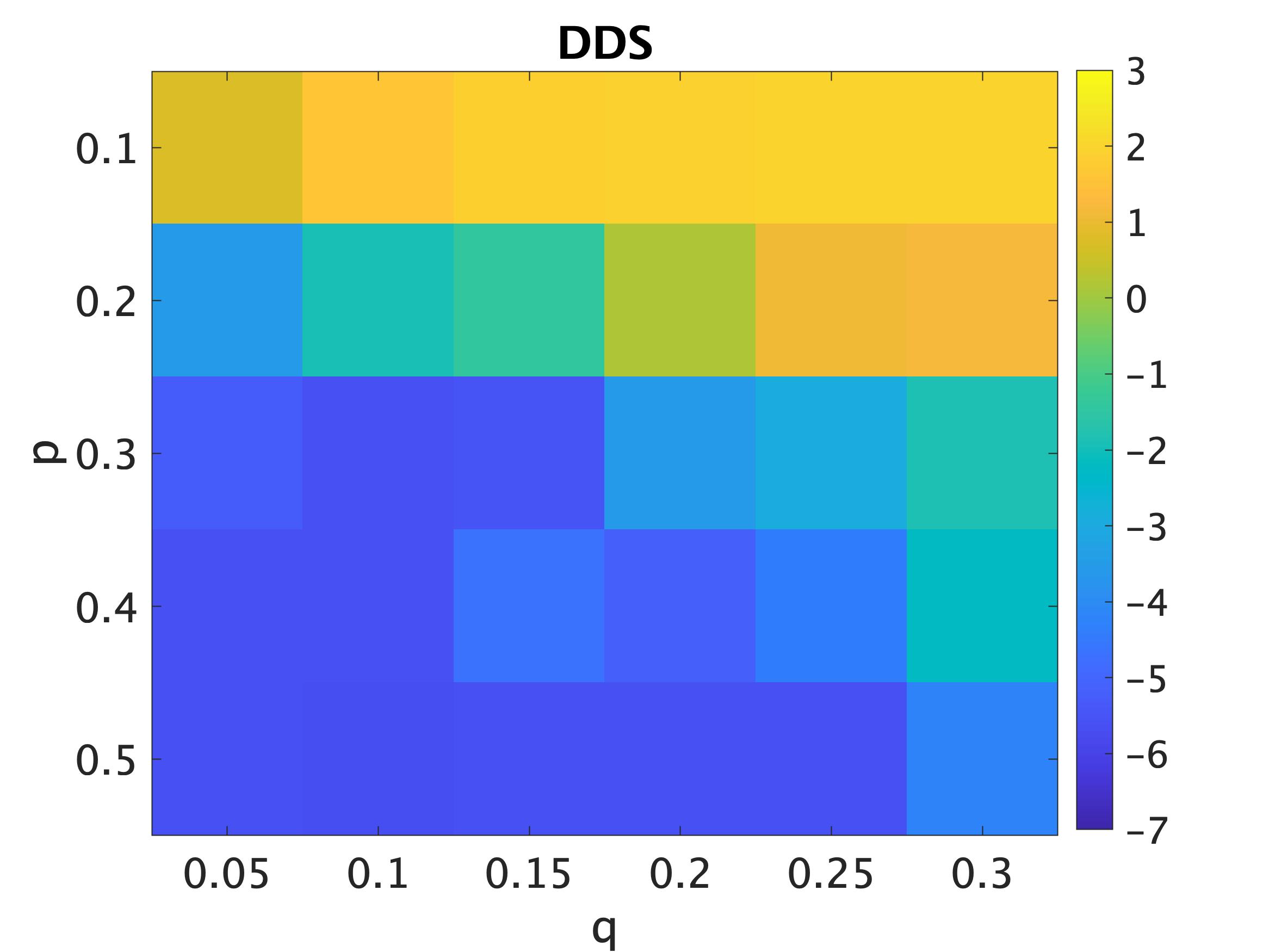}
    \caption{Rotation synchronization experiment with adversarial outliers. Here, $p$ is the parameter of the Erd\"os-R\'enyi graph, and $q$ is the percentage of corrupted edges which are uniformly distributed across this graph. The underlying rotations follow the model in~\eqref{eq:linein}, and the corrupted measurements are pairwise measurements between rotations generated by the separate set~\eqref{eq:lineout}. The color represents the mean of the $\log_{10}$-errors over the 10 generated datasets.}
    \label{fig:exp_lineline_errs}
\end{figure}

 In both experiments, we note that the other competitive algorithms are CEMP and MPLS~\citep{lerman2019robust,shi2020message}. As mentioned earlier, CEMP, and thus also MPLS that uses ideas of CEMP, have higher memory cost than DDS. 
 
% \begin{figure}[ht]
%     \centering
%     \includegraphics[width = .5\textwidth]{plots/times_lineline.jpg}
%     \caption{Rotation synchronization experiment with adversarial outliers. Here, $p$ is the parameter of the Erd\"os-R\'enyi graph, and $q$ is the percentage of corrupted edges which are uniformly distributed across this graph.}
%     \label{fig:exp_lineline_times}
% \end{figure}

}

\section{Conclusion}
\label{sec:conclude}

In this work, we developed the first adversarial robustness guarantees for a multiple rotation averaging algorithm. Our novel algorithm relies on finding descent directions using Tukey depth in the tangent space of $\SO(D)$. To our knowledge, this represents the first application of manifold Tukey depth in an applied setting. In the case of $D=2$ and $D=3$, which most frequently arise in practice, our recovery thresholds are $1/4$ and $1/8$, and the algorithm can be implemented efficiently. % We also discuss some theoretical issues that arise when one tries to analyze the optimization landscape of the least absolute deviations multiple rotation averaging algorithm.
{
We also show how to speed up the algorithm with some approximations, and this approximate algorithm performs competitively on simple synthetic experiments. 
Future work should also examine if it is possible to extend the analysis to the more practical approximate DDS algorithm.}

Another direction for future work is to examine the limits of our analysis. In particular, it would be interesting to know if tighter analyses can yield larger recovery thresholds. At least for the cases of $\SO(2)$ and $\SO(3)$, which arise in applications, the depth descent estimator discussed in this paper has significant recovery thresholds, while also being computationally tractable. { It is not clear what the optimal bounds for recovery with adversarial corruption are in general. Furthermore, if one moves away from adversarial corruption and instead considers special models of data like the uniform corruption model, the bounds could be much better.  } 

Two more concrete directions for future work would be to carry out further examination of the $\zeta$-well-connectedness condition in Assumption~\ref{assump:weakwellconnect}. In particular, it would be interesting to see if it can be relaxed at all,  what its implications are, and when it actually holds. %A concrete way forward might be to explore further its connections with conductance and graph expansion. 

%While we discuss some theoretical issues arising in the analysis of the $L_1$-MRA algorithm, in practice these algorithms tend to perform quite well. Therefore, these methods warrant more study, and in particular, future analysis should determine conditions under which this method is able to converge and recover a set of underlying rotations. It may be that such methods are not stable to adversarial outliers, and so one could assume some more restrictions on the bad edges. It would be interesting to see what sorts of results are achievable in the non-adversarial setting for all of the discussed methods. As inspiration, in the problem of robust subspace recovery, one can take the percentage of corruption to 1 and still obtain exact recovery under very special models of data~\citep{maunu2019well}. It would be interesting to see if nonconvex methods for synchronization enjoy similar guarantees.

Finally, perhaps the most important direction for future work is to give theoretically justified algorithms for a larger range of algorithms employed for SfM~\citep{ozyecsil2017survey,bianco2018evaluating}. Indeed, such theoretical work can lead to new and improved algorithms and also to the development of novel state-of-the-art pipelines.

\bibliography{refs}

\bibliographystyle{plainnat}

%\appendix

\end{document}